\documentclass[11pt]{amsart}

\usepackage{amsmath}
\usepackage{amsthm}
\usepackage{amssymb}
\usepackage{amsfonts}
\usepackage{blkarray}
\usepackage{latexsym, graphicx}
\usepackage[top=1.5in, bottom=1.5in, left=1.5in, right=1.5in]{geometry}
\usepackage{color, enumitem}
\usepackage{tikz}
\usepackage{hyperref}

\usetikzlibrary{arrows}
\usetikzlibrary{decorations.markings,decorations.pathmorphing,arrows}
\pgfarrowsdeclare{mytip}{mytip}{%
  \setlength{\arrowsize}{1pt}
  \addtolength{\arrowsize}{.5\pgflinewidth}
  \pgfarrowsrightextend{0}
  \pgfarrowsleftextend{-5\arrowsize}
}{%
  \setlength{\arrowsize}{1pt}
  \addtolength{\arrowsize}{.5\pgflinewidth}
  \pgfpathmoveto{\pgfpoint{-5\arrowsize}{1.5\arrowsize}}
  \pgfpathlineto{\pgfpointorigin}
  \pgfpathlineto{\pgfpoint{-5\arrowsize}{-1.5\arrowsize}}
  \pgfusepathqstroke
}
\tikzset{myarrow/.style={ decoration={bent,aspect=0.3, markings,mark=at
  position 0.5 with {\arrow[scale=1.2]{latex'}}}, postaction=decorate}}
\tikzset{myarrowshort/.style={ decoration={bent,aspect=0.3, markings,mark=at
  position 0.3 with {\arrow[scale=1.2]{latex'}}}, postaction=decorate}}
\tikzset{myarrowshorter/.style={ decoration={bent,aspect=0.3, markings,mark=at
  position 0.2 with {\arrow[scale=1.2]{latex'}}}, postaction=decorate}}
\tikzset{->-/.style={decoration={markings, mark=at position #1 with
  {\arrow{>}}},postaction={decorate}}}

\tikzset{my_dot/.style={fill, circle, inner sep=0pt,minimum size=1.5pt}}
\tikzset{my_node/.style={fill, circle, inner sep=0pt,minimum size=3pt}}
\tikzset{inv/.style={fill, circle, inner sep=0pt,minimum size=0pt}}

\newtheorem{thm}{Theorem}[section]
\newtheorem{lemma}[thm]{Lemma}
\newtheorem{prop}[thm]{Proposition}
\newtheorem{claim}[thm]{Claim}

\newtheorem{defn}[thm]{Definition}

\newtheorem{Definition}[thm]{Definition}
\newenvironment{definition}
  {\begin{Definition}\rm}{\end{Definition}}

\newtheorem{Example}[thm]{Example}
\newenvironment{example}
  {\begin{Example}\rm}{\end{Example}}

\theoremstyle{remark}
\newtheorem{remark}[thm]{Remark}
\newtheorem{obs}[thm]{Observation}

\newcommand \M {\mathcal{M}}
\newcommand \mgn {M_{g,n}^{\mathrm{trop}}}
\newcommand \mgno {M_{g,n}^{\mathrm{trop}}[1]}
\newcommand \mbargn {\overline{M}_{g,n}^{\mathrm{trop}}}

\newcommand \mtno {M_{2,n}^{\mathrm{trop}}[1]}

\newcommand \dtn {\Delta_{2,n}}
\newcommand \bgn {M_{g,n}^{\mathrm{br}}}  
\newcommand \bgno {M_{g,n}^{\mathrm{br}}[1]}  
  
\newcommand \btno {M_{2,n}^{\mathrm{br}}[1]}

\newcommand \rgno {M^{\mathrm{rep}}_{g,n}[1]}

\newcommand \rtno {M^{\mathrm{rep}}_{2,n}[1]}
\newcommand \Tr {T^{\mathrm{rep}}}
\newcommand \Tnr {T^{\mathrm{nrep}}}
\newcommand \Tbr {T^{\mathrm{br}}}
\newcommand \TTh {T^{\Theta}}
\newcommand \Tcyc {T^{\mathrm{cyc}}}
\newcommand \Tfull {T^{\mathrm{full}}}
\newcommand \Tsig {T^\sigma}
\newcommand \DTT {{DT}^{\Theta}}
\newcommand \G {\mathbf{G}}  
\newcommand \D {\Delta}  
\newcommand \pt {\bullet_{\mathrm{br}}}  
\newcommand \zcell {\epsilon_{\mathrm{br}}}  
\newcommand \Ht {\widetilde{H}} 
\newcommand \Csig {C^\sigma} 
\newcommand{\Gr}{\operatorname{Gr}}
\newcommand{\mcX}{\mathcal{X}}

\newcommand \RR {\mathbb{R}}
\newcommand \ZZ {\mathbb{Z}}

\newcommand \ra {\rightarrow}

\newcommand \lra {\longrightarrow}

\newcommand \tr \textrm
\newcommand \eps \epsilon
\newcommand \sig \sigma

\newcommand{\scr}{\scriptstyle}
\newcommand{\diag}{\operatorname{diag}}
\newcommand{\SNF}{\operatorname{SNF}}
\newcommand{\val}{\operatorname{val}}
\newcommand{\sgn}{\operatorname{sgn}}
\newcommand{\Aut}{\operatorname{Aut}}
\newcommand{\rank}{\operatorname{rank}}
\newcommand{\im}{\operatorname{im}}

\newcommand \Id {\boldsymbol 1} 
\newcommand \ov {\overline}

\newcommand \CC {\mathbb{C}}
\newcommand \QQ {\mathbb{Q}}

\newcommand \del {\partial}
\newcommand \col {\colon}
\newcommand {\set}[1]{\{1,\ldots,#1\}}
\newcommand {\cyc}[2]{\tbinom{#1}{#2}}
\newcommand {\cycl}[2]{\text{\bf -}\!\tbinom{#1}{#2}}
\newcommand {\cycr}[2]{\tbinom{#1}{#2}\!\text{\bf -}}
\newcommand {\cyclr}[2]{\text{\bf -}\!\tbinom{#1}{#2}\!\text{\bf -}}

\title{Topology of the tropical moduli spaces $M_{2,n}$}
\author{Melody Chan}\address{Department of Mathematics, Harvard University, Cambridge, MA 02138}\email{mtchan@math.harvard.edu}

\begin{document}

\begin{abstract} 
We study the topology of the link $\mgno$ of the tropical moduli spaces of curves when $g=2$.  Tropical moduli spaces can be identified with boundary complexes for $\mathcal{M}_{g,n}$, as shown by Abramovich-Caporaso-Payne, so their reduced rational homology encodes top-weight rational cohomology of the complex moduli spaces $\mathcal{M}_{g,n}$.  We prove that $\mtno$ is an $n$-connected topological space whose reduced integral homology is supported in the top two degrees only.  
We compute the reduced Euler characteristic of $\mtno$ for all $n$, and we compute the rational homology of $\mtno$ when $n \le 8$, determining completely the top-weight $\QQ$-cohomology of $\mathcal{M}_{2,n}$ in that range.
\end{abstract}

\maketitle

\section{Introduction}

The tropical moduli spaces of curves $\mgn$ are topological spaces 
that parametrize isomorphism classes of $n$-marked stable abstract tropical curves of genus $g$.  These spaces were introduced in tropical geometry by Brannetti-Melo-Viviani \cite{BMV} and Caporaso \cite{cap13}, building on earlier work of Mikhalkin \cite[\S 5.4]{mik06}.  They have played an important role in many recent advances in tropical geometry; see for example 
\cite{acp, BMV, cap12, cap13, cap14, cha1, cha2, len, viviani}.  
Moreover, tropical moduli spaces have already arisen very naturally in other kinds of geometry.  For example, $M^{\mathrm{trop}}_{0,n}$ is the Billera-Holmes-Vogtmann {\em space of phylogenetic trees} \cite{bhv},
while $M^{\mathrm{trop}}_{g,0}$ is an infinite cone over a compactified quotient of Culler-Vogtmann Outer Space $X_g$ of rank $g$, with the property that the Torelli map on $M^{\mathrm{trop}}_{g,0}$ is compatible with a period map on $X_g$ \cite{owenbaker, cmv}.

While $\mgn$ itself is contractible, since it is a generalized cone complex, the {\em link} of $\mgn$, which we will denote $\mgno$, has interesting topology.  The link can be regarded as the cross-section of $\mgn$ that parametrizes tropical curves of total edge length $1$.  
The work of Abramovich-Caporaso-Payne \cite{acp}, extending the work of Thuillier \cite{thuillier}, identifies $\mgno$ as a {boundary complex} of $\M_{g,n}$.  In fact, their work applies more generally and it identifies the compactification $\mbargn$ canonically with the skeleton of the analytification $\ov{M}_{g,n}^ {\mathrm{an}}$, over any algebraically closed, trivially valued field.  

Boundary complexes, first studied 
by Danilov \cite{danilov}, encode the intersections of the components of a normal crossings compactification.  Furthermore, over $\CC$, they encode the top-weight rational cohomology of the space being compactified.  The work of Deligne \cite{deligne3} shows:
\begin{thm}\label{t:topwtid}
For any smooth, separated Deligne-Mumford stack $\mcX$ of dimension $d$ with a normal crossings compactification $\overline \mcX$ and boundary complex $\Delta(\overline \mcX)$, 
\vspace{-.1cm}
\[
 \Gr_{2d}^W H^{2d-i}(\mcX; \QQ) \  \cong  \ \widetilde H_{i-1}(\Delta(\overline \mcX); \QQ).
\]
\end{thm}
\noindent To be more precise, Deligne's original work is not phrased in the language of stacks, but it extends naturally to that setting.  The details of this formulation can be found in the appendix to \cite{cgp}. In any case the top-weight cohomology of $\M_{g,n}$ is purely combinatorial and is encoded in $\mgno$.

The cases $g=0$ and $g=1$ have been completely settled.  When $g=0$ and $n\ge 3$, the space $M^{\mathrm{trop}}_{0,n}[1]$ has the homotopy type of a wedge of $(n-2)!$ spheres of dimension $n-4$ \cite{bhv}.
The case $g=1$ is settled in \cite{cgp}, where we prove that $M^{\mathrm{trop}}_{1,n}[1]$ is a wedge of $(n-1)!/2$ spheres of dimension $n-1$.  

Here we study the topology of $\mtno$.  Our main theorems are Theorems~\ref{t:1},~\ref{t:2}, and~\ref{t:3} below.
We also prove some intermediate structural results on natural subcomplexes or quotient complexes of $\mtno$:
\begin{itemize}
\item The {\em bridge locus} (Definition~\ref{d:br}) parametrizing tropical curves with bridges, and tropical curves obtained as specializations thereof, is contractible in genus 2 (Theorem~\ref{t:bridge}).
\item For $n \ge 4$, the {\em cyclic theta complex} $C_{2,n}$ (Definition~\ref{d:cyclic}) is $\QQ$-homologous to a point.  Its full $\ZZ$-homology is given in Theorem~\ref{t:cyclic} by
$$\Ht_i(C_{2,n},\ZZ)=\begin{cases}
({\ZZ}/{2\ZZ})^{\frac{(n-1)!}{2}} &\text{if } i = n+1, \\
0 &\text{else.}
\end{cases}$$

\end{itemize}
\noindent These results will imply our first main theorem:
\begin{thm}\label{t:1}
We have
$$\Ht_i(\mtno; \ZZ) = 0$$ 
for $i = 0,\ldots,n$.
Thus, the reduced $\ZZ$-homology of $\mtno$ is concentrated in the top two degrees, for all $n$. Furthermore, $\mtno$ is an $n$-connected space.
\end{thm}

Note that a weaker version of Theorem~\ref{t:1}, for rational instead of integral homology, can be deduced from the fact that the affine stratification number (see~\cite{rv04}) of $M_{2,1}$ is 1.  In fact, a result of Fontanari-Looijenga \cite{fl08} verifies, for $g$ up to $5$, Looijenga's conjecture that the affine stratification number of $M_{g,1}$ is $g-1$.  From Theorem~\ref{t:topwtid} and this fact, it can then be deduced that the reduced rational homology of $\mgno$ is supported in top $g$ degrees whenever Looijenga's conjecture holds.  See~\cite{cgp} for more details.  Note also that the statement on the connectivity of $\mtno$ strengthens, in the genus 2 case, the statement in \cite{cgp} that $\mgno$ is $(n-3)$-connected.

We also prove:

\begin{thm}\label{t:2}
 For all $n\ge 4$, the reduced Euler characteristic of $\mtno$ is
	$$\widetilde{\chi}(\mtno) = (-1)^n\cdot n!/12.$$
\end{thm}

\begin{thm}\label{t:3}
For $n=0,\ldots,8,$ the $\QQ$-homology of $\mtno$ in the top two degrees is given in Table~\ref{table:upto8}.

\begin{table} 
\begin{tabular}{c|ccccccccc}
	\hline
  $n$ &        $0$ & $1$& $2$ & $3$ & $4$ & $5$  &  $6$ & $7$ & $8$\\
  \hline
  $\dim H_{n+2}(\mtno,\QQ)$ &  $0$ & $0$ & $1$ & $0$ & $3$ & $15$ & $86$ & $575$ & $4426$\\
  \hline
  $\dim H_{n+1}(\mtno,\QQ)$ &  $0$ & $0$ &$0$ & $0$ & $1$ & $5$ &  $26$ & $155$ &$1066$ \\
  \hline
 \end{tabular}
 \medskip
  \caption{The rational homology of $\mtno$ for $n\le 8$.}
  \label{table:upto8}
\end{table} 
\end{thm}

The computations obtained in Theorem~\ref{t:3} were done in \texttt{sage} \cite{sage}.  They are explained at the end of Section~\ref{s:mainproof}.   Theorem~\ref{t:3} combined with Theorem~\ref{t:topwtid} determines the top-weight $\QQ$-cohomology of $\mathcal{M}_{2,n}$, for $n\le 8$: it is concentrated in degrees $n+3$ and $n+4$, of ranks given in the table.

The rest of this paper is devoted to proving the main theorems, Theorems~\ref{t:1},~\ref{t:2}, and~\ref{t:3}.  
The paper is largely organized as a sequence of results that shrink pieces of $\mtno$.  First the space is preserved up to homotopy by shrinking a large sublocus, the bridge locus.  Then it is preserved up to rational homology by shrinking another sublocus, the cyclic theta complex.  This will be enough to deduce the main theorems.  The methods in this paper are computational and combinatorial.  The main technical input is the contractibility of the repeated marking subcomplex, established in \cite{cgp}.  The main technical achievements involved are controlling the homology of the bridge locus and the cyclic theta complex completely, using arguments in combinatorial topology and in graph theory.

Here is the outline of the paper.
Following preliminary definitions in Section 2, we show that the bridge locus is contractible in Section 3.  In Section 4, we consider the quotient $\dtn$ by the bridge locus and produce a CW structure on it, which we use in Section 5 to control the homology of a large subcomplex of $\dtn$, the {cyclic theta complex} mentioned previously.  
Then in Section 6 we put the pieces together to prove Theorem~\ref{t:1} and Theorem~\ref{t:3}.  In Section 7, which is independent of Section 6, we prove the Euler characteristic formula in Theorem~\ref{t:2}.

Before diving in, we refer the reader to previous work by Kozlov, who studied a different but highly related space: the moduli space of tropical curves of genus $g$ with $n$ marked points and with {no vertex weights} \cite{kozlov09, kozlov11}.  These have been called {\em pure} tropical moduli spaces elsewhere.  Kozlov's results on the topology of these spaces and the results in this paper seem to be essentially independent.  Basically, allowing nontrivial vertex-weights changes the spaces significantly.

\bigskip
\noindent {\bf Acknowledgments.}  I am grateful to P.~Hersh and S.~Payne for suggesting this direction of research and especially S.~Payne for answering many questions and pointing me to some key references, notably the discussion of affine stratification number.  I thank B.~Sturmfels for his encouragement and several helpful suggestions, and I also thank D.~Cartwright, J.~Conant, S.~Galatius, and K.~Vogtmann for enlightening discussions.  Finally, many thanks to A.~Gaer for computing support at Harvard.  I was supported by NSF DMS-1204278.

\section{The tropical moduli spaces $\mgn$}

In this section we review the general theory of tropical curves, of any genus, and their moduli spaces.  In subsequent sections, we specialize to the case of genus 2.

\subsection{Graph theory}\label{ss:graphs}
In this paper, all graphs will be finite and connected multigraphs, that is, loops and multiple edges are allowed.  We write $V(G)$ and $E(G)$ for the vertex set and the edge multiset of $G$, respectively.

Edges $e,f \in E(G)$ are called {\bf parallel} if they are distinct nonloop edges incident to the same pair of vertices; we write $e \!\parallel\! f$ if so.  If $e \in E(G)$ is not parallel to any edge, we will say that $e$ is a {\bf singleton.}  A {\bf bridge} in $G$ is an edge whose deletion disconnects $G$.  

The {\bf valence} of a vertex $v$, denoted $\val(v)$, is the number of half-edges at $v$; thus loops based at $v$ contribute twice to this sum.
A {\bf cut vertex} in $G$ is a vertex $v\in V(G)$ whose removal disconnects the graph $G$, considered now as a 1-dimensional CW complex.  Our definition of a cut vertex differs slightly from the standard one in graph theory.  As an example, let $R_g$ denote the graph consisting of $g$ loops based at a single vertex $v$.  Then $v$ is a cut vertex if and only if $g\ge 2$.

A vertex-weighted graph is a pair $(G,w)$ where $G$ is a graph and $w\col V(G)\ra \ZZ_{\ge 0}$ is an arbitrary weight function.  Following \cite{ac}, we define the {\bf virtual graph} $G^w$ to be the 
(unweighted) graph obtained from $G$ by adding $w(v)$ distinct loops to each vertex $v$.  We say that a vertex of $v$ is a {\bf virtual cut vertex} of $(G,w)$ if $v$ is a cut vertex of $G^w$.  For example, of the seven vertex-weighted graphs in Figure~\ref{f:genus2}, all but the first have a virtual cut vertex.

\subsection{Tropical curves with marked points}

\begin{definition}\cite{BMV, cap13, mik06}
A {\bf (stable) tropical curve} with $n$ marked points is a quadruple $\Gamma = (G,l,m,w)$ where:

\begin{itemize}
\item $G$ is a graph (as in \S\ref{ss:graphs}),
\item $l\col E(G)\ra \RR_{>0}$ is any function, called a {\em length function}, on the edges,
\item $m\col \set{n} \ra V(G)$ is any function, called a {\em marking function}, and
\item $w\col V(G)\ra \ZZ_{\ge 0}$ is any function; 
\end{itemize}
such the following condition, called the {\em stability condition}, holds: for every vertex $v\in V(G)$, we have
\begin{equation}\label{eq:stability}
2w(v) - 2 + \val(v) + |m^{-1}(v)| > 0.
\end{equation}

\end{definition}
\noindent We remark that our marking function $m$ is combinatorially equivalent to the more common setup of attaching infinite rays to a graph, labeled $\{1,\ldots,n\}$, and we use $m$ just for convenience.

The {\bf genus} of a tropical curve $\Gamma$ is 
\begin{equation}\label{eq:genus}
g(\Gamma) := |E(G)|-|V(G)|+1 + \! \sum_{v\in V(G)} \! w(v).
\end{equation}

\noindent The {\bf total length} of $\Gamma$ is the sum of its edge lengths.  
The total length is called the {\em weight} of the tropical curve elsewhere, but we won't use this terminology since we already have weights on vertices.

\begin{definition}\label{d:type}
Let $\Gamma=(G,l,m,w)$ be a stable tropical curve of genus $g$.

\begin{enumerate}
	\item \label{it:combinatorial}The {\bf combinatorial type} of $\Gamma$ is the triple $(G,m,w)$.  We also define the genus of a combinatorial type as in~\eqref{eq:genus}.  We will write $(G,m,w)=\G$ for short.
We write $T_{g,n}$ for the set of all combinatorial types of tropical curves of genus $g$ with $n$ marked points.  

\item \label{it:unmarked} If $g\ge 2$, we define the {\bf unmarked type} or {\bf underlying type} of $\Gamma$ to be the combinatorial type $(G',w')\in T_{g,0}$ obtained as follows: $(G', w')$ is the smallest connected, vertex-weighted subgraph of $(G,w)$ that contains all cycles of $G$ and all vertices of positive weight.  Here, we assume that we have suppressed all vertices of valence 2, so that $(G',w') $ is stable of genus $g$.
\end{enumerate}
\end{definition}

\begin{example}
There are exactly seven possible unmarked types of tropical curves of genus 2; they are shown in Figure~\ref{f:genus2}.
\end{example}

\begin{figure}
\begin{tikzpicture}[my_node/.style={fill, circle, inner sep=1.75pt}, scale=1]
\begin{scope}
\node[my_node] (A) at (-.5,0){};
\node[my_node] (B) at (.5,0){};
\draw[ultra thick] (A)--(B);
\draw[ultra thick] (A) to [bend right=90] (B);
\draw[ultra thick] (A) to [bend left=90] (B);
\draw (0,-1) node {I};
\end{scope}
\begin{scope}[shift = {(2,0)}]
\node[my_node] (A) at (-.25,0){};
\node[my_node] (B) at (.25,0){};
\draw[ultra thick] (A)--(B);
\draw[ultra thick] (A) to [out = 225, in = 135, looseness=20] (A);
\draw[ultra thick] (B) to [out = 45, in = 315, looseness=20] (B);
\draw (0,-1) node {II};
\end{scope}
\begin{scope}[shift = {(4,0)}]
\node[my_node] (A) at (0,0){};
\draw[ultra thick] (A) to [out = 225, in = 135, looseness=20] (A);
\draw[ultra thick] (A) to [out = 45, in = 315, looseness=20] (A);
\draw (0,-1) node {III};
\end{scope}
\begin{scope}[shift = {(6,0)}]
\node[my_node] (A) at (-.25,0){};
\node[my_node, label=right:$1$] (B) at (.25,0){};
\draw[ultra thick] (A)--(B);
\draw[ultra thick] (A) to [out = 225, in = 135, looseness=20] (A);
\draw (0,-1) node {IV};
\end{scope}
\begin{scope}[shift = {(7.5,0)}]
\node[my_node, label=right:$1$] (A) at (.25,0){};
\draw[ultra thick] (A) to [out = 225, in = 135, looseness=20] (A);
\draw (0.25,-1) node {V};
\end{scope}
\begin{scope}[shift = {(9.5,0)}]
\node[my_node, label=left:$1$] (A) at (-.25,0){};
\node[my_node, label=right:$1$] (B) at (.25,0){};
\draw[ultra thick] (A)--(B);
\draw (0,-1) node {VI};
\end{scope}
\begin{scope}[shift = {(11,0)}]
\node[my_node, label=right:$2$] (B) at (0,0){};
\draw (0,-1) node {VII};
\end{scope}
\end{tikzpicture}
\caption{The seven combinatorial types in $T_{2,0}$. The vertices have weight zero unless otherwise indicated.}\label{f:genus2}
\end{figure}
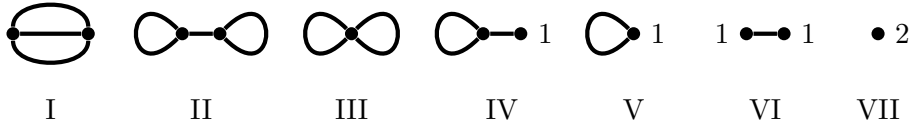

\begin{remark}\label{r:nonempty}
  There is an obvious notion of isomorphism of combinatorial types $(G_1,m_1,w_1) \cong (G_2,m_2,w_2)$, i.e.~a multigraph isomorphism between $G_1$ and $G_2$  carrying $m_1$ to $m_2$ and $w_1$ to $w_2$.
Then there are only finitely many types in $T_{g,n}$ up to isomorphism.  
Furthermore, $T_{g,n}$ is nonempty if and only if $2g-2+n>0$, due to the stability condition~\eqref{eq:stability}.  
\end{remark}

Let $\G=(G,m,w) \in T_{g,n}$, and let $e$ be any edge of $G$.  The {\bf contraction} $\G/e$ is defined to be the combinatorial type $\G' = (G',m',w')\in T_{g,n}$ obtained  as follows: if $e$ is a loop based at $v\in V(G)$, then $\G'$ is obtained by deleting $e$ and increasing $w(v)$ by 1.  Otherwise, if $e$ has endpoints $v_1$ and $v_2$, then $\G'$ is obtained by deleting $e$ and identifying $v_1$ and $v_2$ into a new vertex $v$.  We set $w'(v) = w(v_1) + w(v_2)$.

Now for any subset $S\subseteq E(G)$, we define $\G/S$ to be the contraction of all edges in $S$.  It is well-known that they can be contracted in any order.  
Note that the new type $(G,m,w)/S$ also lies in $T_{g,n}$.

\subsection{Construction of $\mgn$}  Given $\G= (G,m,w) \in T_{g,n}$, write
\begin{eqnarray*}
\ov{C}(\G) &:=& \RR_{\ge 0}^{E(G)}, \\
\D(\G) &:=& \{l\col E(G)\ra \RR_{\ge 0}\}~|~\sum{ l(e)} = 1\} \subset \ov{C}(\G).
\end{eqnarray*}
So $\D(\G)$ is a simplex of dimension $|E(G)|\!-\!1$ and $\ov{C}(\G)$ is an infinite cone over it.

Now given types $\G, \G'\in T_{g,n}$ and an isomorphism 
$$\alpha\col \G' \stackrel{\cong}{\lra}\G/S$$
for some $S \subseteq E(G)$, 
we associate the linear map $$L_\alpha\col \ov{C}(\G')\ra \ov{C}(\G)$$
that identifies $\ov{C}(\G')$ with the face of $\ov{C}(\G)$ consisting of those $l\in \RR_{\ge 0}^{E(G)}$ that are zero on $S$.  That is, it sends $l'\col E(G')\ra \RR_{\ge 0}$ to $l\col E(G) \ra \RR_{\ge 0}$ with
$$l(e) = \begin{cases}
0 &\textrm{if }e \in S,\\
l'(\alpha^{-1}(e)) &\textrm{otherwise.}
\end{cases}
$$

Call such a linear map $L_\alpha$ a {\em face identification.}
Note that we specifically allow $S$ to be empty, in which case a nontrivial automorphism of $\G$ produces a map $\ov{C}(\G)\ra\ov{C}(\G)$.  

\begin{defn}\label{d:mgn}\cite{BMV, cap13}
The {\bf moduli space of tropical curves}, denoted $\mgn$, is the colimit in the category of topological spaces
$$ \lim_{\lra} (\{\ov{C}(\G)\}, \{L_\alpha\})$$
as $\G$ ranges over $T_{g,n}$ and $\{L_\alpha\}$ is the set of all face identifications.
\end{defn}

\noindent Note that the points of $\mgn$ are in bijection with isomorphism classes of tropical curves of genus $g$ and $n$ marked points.  It suffices for our purposes to simply regard $\mgn$ as a topological space; it can be given extra structure, e.g.~that of a stacky polyhedral fan, as in \cite{cmv}.

\begin{defn}\label{d:mgno}
Let $\mgno$ denote the subset of $\mgn\!$ parametrizing tropical curves of total length 1. Thus
$$\mgno = \lim_{\lra} (\{\D(\G)\}, \{L_\alpha\}),$$
where by abuse of notation we write $L_\alpha$ for the restrictions of the maps $L_\alpha$ to the simplices $\D(\G)$.  
\end{defn}

\noindent Thus  $\mgn$ is an infinite cone over $\mgno$; the cone point is the unique tropical curve in $\mgn$ that is a single vertex of weight $g$ and $n$ markings.

\begin{defn}
We say that a combinatorial type $(G,m,w)$ or a tropical curve $(G,l,m,w)$ 
has {\bf repeated markings} if the marking function $m$ is not injective.  Write $\Tr_{g,n}$ and $\Tnr_{g,n}$ for the set of repeating and nonrepeating types, respectively, in $T_{g,n}$.
Call the subset of $\mgno$ of  tropical curves with repeated markings the {\bf repeated marking locus}, denoted $\rgno$.
\end{defn}  

If $(G,m,w)$ has repeated markings then any contraction $(G,m,w)/S$ does too.  
It follows that $\rgno$ is closed.  The following result, proved in \cite{cgp}, will be used in the next section:
\begin{thm}\cite{cgp}\label{t:repcollapse}
For all $g > 0$ and $n>1$, $\rgno$ is contractible.
\end{thm}

\section{The bridge locus is contractible in genus 2}  \label{s:bridge}

We now define the bridge locus, a large subset of $\mgno$ consisting of the closure of all tropical curves with bridges.  The main result of this section 
is that when $g=2$, the bridge locus is contractible.

\begin{definition}\label{d:br}
Fix $g, n\ge 0$ with $2g-2+n>0$. We define $\bgn$ to be the closure in $\mgn$ of the set of tropical curves containing a bridge.  We let $\bgno$ be the closed subset of $\bgn$ of curves of total length 1, called the {\bf bridge locus}.
\end{definition}

Let $\Tbr_{g,n}$ denote the combinatorial types of curves occurring in the bridge locus.  That is, $\Tbr_{g,n}$ is the closure under the contraction operation of the combinatorial types that have bridges.

\begin{lemma} \label{l:br}
Fix $g, n\ge 0$ with $2g-2+n>0$.  
A type $\G\in T_{g,n}$ lies in $\Tbr_{g,n}$ if and only if 
\begin{enumerate}
	\item \label{it:inrep} $\G\in \Tr_{g,n}$, or
	\item \label{it:virtualcut} $\G$ has a virtual cut vertex (see \S\ref{ss:graphs}).
\end{enumerate}
\end{lemma}

\begin{proof}
Suppose ${\mathbf H}\in \Tbr_{g,n}$.  Then ${\mathbf H}$ is a contraction of some type $\G=(G,m,w) \in \Tbr_{g,n}$ that has a bridge, say $e=v_1v_2\in E(G)$.  We will show that (1) or (2) holds for ${\mathbf H}$ by analyzing $\G$.
Basically, if $e$ separates two subcurves of positive genus, then $\G$ and hence ${\mathbf H}$ must have a virtual cut vertex. If instead one subcurve has genus 0 then $\G$ and hence ${\mathbf H}$ must be repeating.
 
More formally: consider the following operation on $\G$.  We delete $e$ and add a new marked point to each of $v_1$ and $v_2$, obtaining two types $\G_1 = (G_1,m_1,w_1)$ and $\G_2 = (G_2,m_2,w_2)$, which are each stable by the stability condition~\eqref{eq:stability}.  

Now, if $g(\G_i)>0$ for both $i$, then $e$ must have been a bridge in the unmarked type for $\G$. Then any contraction of $\G$ has a virtual cut vertex.
Otherwise, one of the graphs, say $G_1$, is a tree.
If $G_1 = \{v_1\}$ then it supports at least three markings, of which at least two are original to $\G$.  Otherwise, consider any leaf of $G_1$ other than $v_1$; then it supports at least two markings in $G$.  So $\G$ was repeating, and any contraction of $\G$ must be repeating.  This proves that if ${\mathbf H}\in \Tbr_{g,n}$ then \eqref{it:inrep} or~\eqref{it:virtualcut} holds.

For the converse, suppose that ${\mathbf H}$ has a virtual cut vertex. Then this vertex can be expanded into a bridge that separates two types whose genera sum to $g$.  Similarly, if ${\mathbf H}\in \Tr_{g,n}$, then ${\mathbf H}$ has a vertex $v$ supporting markings $i\ne j$.  So the vertex $v$ can be expanded into a bridge that separates markings $i$ and $j$ from the rest of the curve.

\end{proof}

Now we specialize to the case $g=2$.

\begin{thm}\label{t:bridge}
For each $n\ge 0,$ the bridge locus $\btno$ is contractible.
\end{thm}

\begin{proof}
We set some terminology.  
Suppose $\G=(G,m,w)\in\Tbr_{2,n}$ is nonrepeating, and therefore has a virtual cut vertex by Lemma~\ref{l:br}.  By the classification in Figure~\ref{f:genus2}, the bridges in $\G$, if any, form a single path that we will denote $P=P(\G)$, with vertices $p_0,\ldots, p_k$, throughout.  For example, in Figure~\ref{fig:stretching}, the paths $P$ have lengths $4,3,3,$ and $2$.  By convention, if $\G$ has no bridges then we will take $P=p_0$ to be the unique virtual cut vertex in $\G$.  Each interior vertex $p_1,\ldots, p_{k-1}$ necessarily supports exactly one marked point, while $p_0$ and $p_k$ support zero or one marked points each. 
We extend this definition to any genus 2, nonrepeating $n$-marked tropical curve $\Gamma$ with a virtual cut vertex, defining $P=P(\Gamma)$ to be the path formed by all the bridges of $\Gamma$, or the unique virtual cut vertex of $\Gamma$ if $\Gamma$ has no bridges.  

Suppose again that $\G\in \Tbr_{2,n}$ is nonrepeating.  Say that $\G$ is an {\em expanded} type if $P(\G)$ has positive length and both $p_0$ and $p_k$ are unmarked.  If $\G$ is an expanded type, call the first and last edges of $P$ {\em expanding}, and call all the other edges of $\G$ {\em shrinking}.  So an expanded type $\G$ has one expanding edge if $P(\G)$ has length 1, and two expanding edges otherwise.  See Figure~\ref{fig:stretching}.    Notice that by contracting one or both of the expanding edges of any expanded type, we recover all nonrepeating types in $\Tbr_{2,n}$: thus
\begin{obs}\label{o:expanded}
Every nonrepeating type $\G'\in \Tbr_{2,n}$ is obtained from an expanded type $\G\in \Tbr_{2,n}$ by contracting one or two expanding edges of $\G$.  
\end{obs}
\noindent There is no structural meaning to the terms expanding or shrinking, except in relation to the deformation retracts we are about to define.
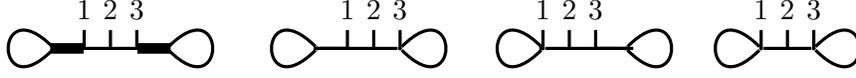
\begin{figure}
\begin{tikzpicture}[my_node/.style={fill, circle, inner sep=1.75pt}, scale=1]
\begin{scope}[shift = {(0,0)}]
\node[my_dot] (A) at (-.75,0){};
\node[my_dot] (B) at (.75,0){};
\node[inv] (m1) at (-.35,0){};
\node[inv, label=90:$1$] (n1) at (-.35,.25){};
\draw[very thick] (m1)--(n1);
\node[inv] (m2) at (0,0){};
\node[inv, label=90:$2$] (n2) at (0,.25){};
\draw[very thick] (m2)--(n2);
\node[inv] (m3) at (.35,0){};
\node[inv, label=90:$3$] (n3) at (.35,.25){};
\draw[very thick] (m3)--(n3);
\draw[line width=4] (-.8,0)--(m1);
\draw[line width=4] (.8,0)--(m3);
\draw[very thick] (A)--(B);
\draw[very thick] (A) to [out = 225, in = 135, looseness=60] (A);
\draw[very thick] (B) to [out = 45, in = 315, looseness=60] (B);
\end{scope}
\begin{scope}[shift = {(6.5,0)}]
\node[my_dot] (A) at (-.75,0){};
\node[my_dot] (B) at (.35,0){};
\node[inv] (m1) at (-.75,0){};
\node[inv, label=90:$1$] (n1) at (-.75,.25){};
\draw[very thick] (m1)--(n1);
\node[inv] (m2) at (-.4,0){};
\node[inv, label=90:$2$] (n2) at (-.4,.25){};
\draw[very thick] (m2)--(n2);
\node[inv] (m3) at (-.05,0){};
\node[inv, label=90:$3$] (n3) at (-.05,.25){};
\draw[very thick] (m3)--(n3);
\draw[very thick] (-.8,0)--(m1);
\draw[very thick] (.45,0)--(m3);
\draw[very thick] (A)--(B);
\draw[very thick] (A) to [out = 225, in = 135, looseness=60] (A);
\draw[very thick] (B) to [out = 45, in = 315, looseness=60] (B);
\end{scope}
\begin{scope}[shift = {(3.5,0)}]
\node[my_dot] (A) at (-.75,0){};
\node[my_dot] (B) at (.35,0){};
\node[inv] (m1) at (-.35,0){};
\node[inv, label=90:$1$] (n1) at (-.35,.25){};
\draw[very thick] (m1)--(n1);
\node[inv] (m2) at (0,0){};
\node[inv, label=90:$2$] (n2) at (0,.25){};
\draw[very thick] (m2)--(n2);
\node[inv] (m3) at (.35,0){};
\node[inv, label=90:$3$] (n3) at (.35,.25){};
\draw[very thick] (m3)--(n3);
\draw[very thick] (-.8,0)--(m1);
\draw[very thick] (A)--(B);
\draw[very thick] (A) to [out = 225, in = 135, looseness=60] (A);
\draw[very thick] (B) to [out = 45, in = 315, looseness=60] (B);
\end{scope}
\begin{scope}[shift = {(9,0)}]
\node[my_dot] (A) at (-.35,0){};
\node[my_dot] (B) at (.35,0){};
\node[inv] (m1) at (-.35,0){};
\node[inv, label=90:$1$] (n1) at (-.35,.25){};
\draw[very thick] (m1)--(n1);
\node[inv] (m2) at (0,0){};
\node[inv, label=90:$2$] (n2) at (0,.25){};
\draw[very thick] (m2)--(n2);
\node[inv] (m3) at (.35,0){};
\node[inv, label=90:$3$] (n3) at (.35,.25){};
\draw[very thick] (m3)--(n3);
\draw[very thick] (A)--(B);
\draw[very thick] (A) to [out = 225, in = 135, looseness=60] (A);
\draw[very thick] (B) to [out = 45, in = 315, looseness=60] (B);
\end{scope}

\end{tikzpicture}
\caption{Four examples of types in the bridge locus $\Tbr_{2,3}$.  The first type is expanded, and its expanding edges are marked in bold. The other three are obtained from it by contracting one or both expanding edges.}
\label{fig:stretching}
\end{figure}

Namely, for each $i\ge 0$, let
\smallskip
$$X_i = \rtno \cup \{\Gamma\in \btno\setminus\rtno : |m^{-1}(P)| + |w^{-1}(P)| \ge i\}$$

\noindent where $P=P(\Gamma)$ is the path in $\Gamma$ defined above.  The condition above says that the number of markings on $P$ plus the total weight supported on $P$ is at least $i$.  Again, the fact that $\Gamma \in \btno\setminus \rtno$ has a virtual cut vertex at all, and hence a well-defined path $P(\Gamma)$, is guaranteed by Lemma~\ref{l:br}. 

Note $X_0 = \btno$.
Now we will give a map, for each $i\ge 0$,
\begin{equation}
\rho_i\col X_i \times [0,1]\lra X_i
\label{eq:rhoi}
\end{equation}
and we will show that it is a strong deformation retract from $X_i$ to $X_{i+1}$.  This will yield that $\btno$ is homotopy equivalent to $X_{n+2}$, and we will verify later that $X_{n+2}$ is contractible.

We define $\rho_i$ as follows.  
First, we declare that $\rho_i$ fixes $X_{i+1}$ pointwise for all time.
Next, for every expanded nonrepeating type $\G=(G,m,w)\in \btno$ with $|m^{-1}(P)| + |w^{-1}(P)| = i$, let $j^+$ be the number of expanding edges in $G$ and $j^- = |E(G)|-j^+$ be the number of shrinking edges.  So $j^+=1$ or $j^+=2$ as observed above.   As an intermediate step to defining $\rho_i$, we define a map
$$\rho_{i,\G}\col \D(\G) \times [0,1] \ra \D(\G)$$
as follows.  Given $(l\col E(G)\ra \RR_{\ge0}) \in \D(\G)$, let 
\begin{equation}\label{eq:c}
c = \min \left \{\,l(e)~|~ e \text{ is a shrinking edge of }G\right\}
\end{equation}
\noindent or set $c=0$ if $G$ has no shrinking edges.  
Then we define
$\rho_{i,\G}(l, t)\col E(G)\rightarrow \RR_{\ge 0}$ by
$$\rho_\G(l, t)(e) = \begin{cases}
l(e)-tc &\text{if $e$ is a shrinking edge} \\
l(e)+tc\frac{j^-}{j^+} &\text{if $e$ is an expanding edge}.
\end{cases}
$$
The map $\rho_{i,\G}$ has the following properties.  It is clearly continuous, and it clearly descends to a well-defined map on tropical curves represented by length functions in $\Delta(\G)$.  Furthermore, it fixes, at all times, a point $l$ if and only if the quantity $c$  in ~\eqref{eq:c} is zero.  

We claim that the maps $\rho_{i,\G}$, along with the map $X_{i+1}\times [0,1]\ra X_{i+1}$ fixing $X_{i+1}$, glue  to give a continuous map $\rho_i\col X_i \times [0,1] \ra X_i$.
To check this, let $\Gamma\in X_i$ be a tropical curve.
First, suppose $\Gamma \in X_{i+1}$.  Suppose $\Gamma \in \iota_\G(\Delta(\G))$, where $\G$ is an expanded, nonrepeating type in $\Tbr_{2,n}$ with 
$$|m^{-1}(P(\G))|+|w^{-1}(P(\G))| = i,$$
and $\iota_\G$ denotes the canonical map $\Delta(\G) \ra \mgno$. Then we want to show that $\Gamma$ is fixed by $\rho_{i,\G}$ for all time.  

Indeed, if $\Gamma = \iota_\G(l)$ for $l\in \Delta(\G)$, then the type $\G'$ underlying $\Gamma$ was obtained from $\G$ by contraction of some edges.  This contraction is nontrivial, since 
$$|m^{-1}(P(\G'))|+|w^{-1}(P(\G'))| > i = |m^{-1}(P(\G))|+|w^{-1}(P(\G))|.$$ 
The key observation is that a contraction involving only the {\em expanding} edges of $\G$ never produces a repeating type, and, it never increases $|m^{-1}(P)|$ or $|w^{-1}(P)|$.  Therefore the contraction producing $\G'$ from $\G$ must involve at least one shrinking edge.  This means that the quantity~\eqref{eq:c} is zero, so that $\Gamma$ is fixed under $\rho_{i,\G}$.  

Second, suppose $\Gamma \in X_i\setminus X_{i+1}$.  So $\Gamma\in\btno$ is a nonrepeating tropical curve with
$$|m^{-1}(P(\Gamma))| + |w^{-1}(P(\Gamma))| = i.$$  We just need to verify that $\Gamma \in \iota_\G(\Delta(\G))$ for a unique expanded, nonrepeating type $\G$ with 
$$|m^{-1}(P(\G))| + |w^{-1}(P(\G))| = i.$$ 
The existence of $\G$ follows from Observation~\ref{o:expanded}, and the fact that contraction of expanding edges does not change $|m^{-1}(P)|$ or $|w^{-1}(P)|$.  The expanded type $\G$ is furthermore unique, as any other uncontraction of $\G'$ either decreases $|m^{-1}(P)|$, by moving a marked point off an endpoint of $P$ and onto a cycle; or decreases $|w^{-1}(P)|$, by creating a new loop at an endpoint of $P$.

We have thus constructed a continuous map $\rho_i\colon X_i\times [0,1]\ra X_i$, with the property that $X_{i+1}$ is fixed.  Furthermore $\rho_i(X_i,1) \subseteq X_{i+1}$.  Indeed let $\Gamma\in X_{i}\setminus X_{i+1}$; then $\rho_i(\Gamma,1)$ is, by construction, obtained by contracting at least one shrinking edge $e$ of $\Gamma.$  Now, it is not possible that both endpoints of $e$ are unmarked, for then $e = P(\Gamma)$ would be the unique expanding edge of $\Gamma$.  Suppose exactly one vertex of $e$ is unmarked; that vertex must be an endpoint of $P(\Gamma)$, say $p_0$.  Since $e$ is not itself a bridge (for then it would be an expanding edge), it must lie on a cycle of $\Gamma$.  Since the vertex of $e$ different from $p_0$ is marked, it follows that contracting $e$ moves a marking onto $P$.  Therefore $\rho_i(\Gamma,1)\in X_{i+1}$ as desired.  Finally, suppose both endpoints of $e$ are marked.  Then either $e$ is a loop and contracting it increases the weight supported on $P$; or, $e$ is a nonloop and contracting it produces a repeated marking.  We conclude that $\rho_i$ is a deformation retract of $X_i$ onto $X_{i+1}$.

Now, stringing together the maps $\rho_i$ produces a deformation retract of $\btno$ onto $X_{n+2}$.  We claim $X_{n+2}$ is contractible, and this will finish the proof.  Indeed $X_{n+2}$ consists of $\rtno$ along with all nonrepeating tropical curves obtained from the one-edged curve (VI) in Figure~\ref{f:genus2} by adding marked points at $n$ distinct places.  Thus if $n=0$ or $n=1$ then $X_{n+2}$ is obviously contractible.  
If $n\ge 2$, then one may retract $X_{n+2}$ onto $\rtno$ in a similar fashion to the argument above.  We consider all nonrepeating expanded types $\G$ of curves in $X_{n+2}$, and consider the map that expands the first and last edges of the path $P(\G)$ and shrinks the interior edges of $P(\G)$.  Note that there is at least one interior edge since $n\ge 2$ by assumption.  These maps $ \Delta(\G)\times[0,1]\ra \Delta(\G)$ glue together just as above into a deformation retract of $X_{n+2}$ onto $\rtno$.  Finally $\rtno$ is contractible, by \cite{cgp}.

\end{proof}

\section{A CW structure on the quotient by the bridge locus}

We will take $g=2$ throughout the rest of this paper.  Having shown that $\btno$ is a contractible subspace of $\mtno$ in Theorem~\ref{t:bridge}, we now study the quotient
$$\dtn :=\mtno/\btno.$$
Now, the space $\mtno$ can certainly be given a CW complex structure in which $\btno$ is a subcomplex.  For example, once can first take a barycentric subdivision of the simplices $\D(\G)$.  Then $\mtno$ is in fact a $\Delta$-complex on the barycentric cells.  (It is not a {\em simplicial} complex on the barycentric cells unless $n=0$.  For an example, consider the triangle corresponding to the graph with vertices $\{x,y\}$, edges $\{xy, xy, yy\}$, no vertex weights, and a single marking at $x$.)  In any case, by e.g.~\cite[Proposition 0.17, A.3]{hatcher}, we have:
\begin{obs}\label{o:hausdorff}
$\dtn$ is a CW complex, hence Hausdorff. By Theorem~\ref{t:bridge}, it is homotopy equivalent to $\mtno$.
\end{obs}

We will immediately dispense with the barycentric cell structure on $\mtno$, because it is impractical.  Instead, in this section we will equip   $\dtn$ with a CW structure that is useful for  computations.  It will take advantage of the fact that the tropical curves in the part of $\mtno$ that has not been collapsed have very limited automorphisms, namely, as long as $n\ge 4$, at most a single transposition of two parallel edges.  So, in what follows, we take $n\ge 4.$  The cases $n\le 3$ are small enough to be handled directly, as in Section~\ref{s:mainproof}.

Write $\TTh_{2,n}$ for those nonrepeating types $\G \in T_{2,n}$ whose underlying type is a theta graph, i.e.~of type I in Figure~\ref{f:genus2}.  We call these {\em theta types}.  We emphasize that by definition they must be nonrepeating.  See Figure~\ref{fig:createparallel} for some examples.    

\begin{lemma}\label{l:thetasleft}
We have
$$T_{2,n} = \Tbr_{2,n} \amalg \TTh_{2,n}.$$  
\end{lemma}

\begin{proof}
This is immediate from Lemma~\ref{l:br}, noting that of the seven types in $T_{2,0}$, shown in Figure~\ref{f:genus2}, all but type I have a virtual cut vertex.
\end{proof}

Now we characterize the automorphisms of theta types.  Note that every vertex of a theta type must have weight zero, so we write $\G=(G,m,0) = (G,m)$ for short.
\begin{lemma}\label{c:auts}
Let $n\ge 4$ and let $\G = (G,m) \in \TTh_{2,n}$.
Then $\Aut(\G)$ is either trivial or is generated by a single transposition of parallel edges.
\end{lemma}

\begin{proof}
Consider first an unmarked graph of type I, which we will henceforth call $\Theta$ as a memory aid.  Let $e_1,e_2,e_3$ denote the three edges of $\Theta$. Now think
of constructing $\G$ by adding $n$ markings at distinct points anywhere on $\Theta$, i.e.~either on the two vertices of $\Theta$ or in the interiors of the three edges.  Formally, of course, adding a marked point to the interior of an edge should be viewed as first subdividing the edge, then marking the new vertex.

First suppose some $e_i$ had its interior marked twice.  Then any automorphism of $\G$ must fix pointwise the path corresponding to $e_i$, including its endpoints.  Then $\G$ has at most an automorphism swapping parallel edges: this automorphism would occur precisely if all $n$ markings were to lie on a single $e_i$.  

Now suppose instead that each $e_i$ has its interior marked at most once.  Then $n\ge 4$ implies that one or both of the vertices of $\Theta$ were marked.  Thus any automorphism of $\G$ must fix its two 3-valent vertices.  Furthermore, at least two of the edges of $\Theta$, say $e_1$ and  $e_2$, have their interiors marked; so any automorphism of $\G$ must fix the two paths corresponding to $e_1$ and $e_2$.  Hence all of $\G$ must be fixed.  So in this case, $\G$ has trivial automorphism group.

\end{proof}

\begin{definition}\label{d:dec}
Let $\G =(G,m)\in \TTh_{2,n}$ with $n\ge 4$.  
\begin{enumerate}
	\item A {\bf decoration} $\delta$ on $\G$ is a choice, for any pair of parallel edges $e$ and $ e'$ in $G$, of a formal relation 
$$e \le_\delta e' \qquad\textrm{or}\qquad e =_\delta e' \qquad\textrm{or}\qquad e \ge_\delta e'.$$  
\item A {\bf decorated type} is a pair $(\G,\delta).$  If $\Aut(\G)$ is trivial, then $\delta$ is trivial and we may simply write $\G$ instead of $(\G,\delta)$.
An isomorphism of decorated types $(\G,\delta)$ and $(\G', \delta')$ is an isomorphism $\G \cong \G'$ carrying $\delta$ to $\delta'$.    
Write $\DTT_{2,n}$ for the set of (isomorphism classes of) decorated types.
\item For each $(\G, \delta)\in \DTT_{2,n}$, we let 
$$\D(\G,\delta) \subset \RR^{E(G)}$$
be the polytope parametrizing those $l\col E(G)\rightarrow \RR_{\ge 0} \in \D(\G)$  such that if~$e$ and $e'$ are parallel  then  $l(e)$ and $l(e')$ satisfy the relation given by $\delta$.
\end{enumerate}
\end{definition}

 For any $e\in E(G)$, write $l_e$ for the indicator length function for $e$, i.e.~$l_e(e) = 1$ and $l_e(e') = 0$ for $e'\ne e.$  

\begin{lemma}\label{l:faces}
Let $(\G,\delta)\in \DTT_{2,n}$ be any decorated type.
\begin{enumerate}
	\item \label{it:dimdec}
	$\D(\G,\delta)$ is a simplex of dimension $|E(G)| - 1 - $ \#equal signs in $\delta.$
	\item \label{it:vertdec}
	The vertices of $\D(\G,\delta)$ are
	$$\left\{l_e~|~ e\text{ is a singleton or } e \ge_\delta e'\right\} \cup \left\{\frac{1}{2}(l_e + l_{e'})~|~ e =_\delta e'\text{ or }e \le_\delta e'\right\}.$$
	\item \label{it:facetdec}
	The facets of $\D(\G,\delta)$ are given by equations
\begin{eqnarray*}
\left\{l(e)=0 ~|~ e\text{ is a singleton or } e \le_\delta e'\right\} &\cup& \left\{l(e)=l(e')= 0~|~ e =_\delta e'\right\} \\
&\cup& \left\{l(e) = l(e')~|~ e \ge_\delta e'\right\}.
\end{eqnarray*}
\item \label{it:relint} A point $l\col E(G)\ra \RR_{\ge 0}$ is in the relative interior of $\D(\G,\delta)$ if and only if $l(e) \ne 0$ for each $e\in E(G)$, and further
$$ l(e) < l(e') \quad \text{if} \quad e\le_\delta e' \qquad \text{and} \qquad l(e) = l(e') \quad \text{if}\quad  e =_\delta e'.$$
\end{enumerate}
\end{lemma}

\begin{proof}
From Definition~\ref{d:dec}, it follows that the points in $\Delta(\G,\delta)$ are exactly those that can be expressed as a convex combination of the finitely many points listed in part~\eqref{it:vertdec}, and that this expression is unique if it exists.  
Thus $\D(\G,\delta)$ is a simplex.  It has a vertex for every singleton $e$, two vertices for every $e\parallel e'$ with $e\le_\delta e'$ or $e\ge_\delta e'$, and a single vertex for every $e\parallel e'$ with $e=_\delta e'$.  Thus the number of vertices in $\D(\G,\delta)$ is $|E(G)|~-$ \#equal signs in $\delta$.  
Next, one may check that each of the linear slices of $\D(\G,\delta)$ listed in~\eqref{it:facetdec} contains all vertices but one, and that they are all distinct.  Finally, the points described in~\eqref{it:relint} are precisely those points of $\D(\G,\delta)$ not contained in any facet.
\end{proof}

Let $n\ge 4.$  We now describe a CW structure on $\dtn = \mtno/\btno$.  We will give a finite number of maps of simplices into $\dtn$, and then prove that they are characteristic maps.

First, let $\pt$ be a $0$-simplex, i.e.~a point, and consider the map $\phi_{\text{br}}\col \pt\ra \dtn$ sending $\pt$ to the collapse of $\btno$.  Next, for each $(\G,\delta)\in \DTT_{g,n}$, consider the map 
$$\phi_{\G,\delta}\col \D(\G,\delta)\lra \D(\G) \lra \mtno \lra \dtn$$
that factors through the inclusion of polytopes $\D(\G,\delta)\subseteq \D(\G)$ and the canonical map $\iota_\G\col \D(\G) \ra \mtno$.

\begin{prop}\label{p:cw}
Let $n\ge 4.$  The maps $\phi_{\text{br}}$ and $\{\phi_{\G,\delta}~|~ (\G,\delta)\in \DTT_{2,n}\}$
are the characteristic maps for a CW complex structure on $\dtn$. 
\end{prop}
\begin{proof}
We will show two claims:
\begin{enumerate}
	\item the maps $\phi_{\text{br}}$ and $\{\phi_{\G,\delta}\}$, restricted to the interiors of their domains, are homeomorphisms onto their images.  We call the images of these interiors {\em cells} and denote them  $\eps_\text{br}$ and $\eps_{\G,\delta}$.  Furthermore, we have 
	\begin{equation} \label{eq:setequality}
\dtn = \left( \coprod \eps_{\G,\delta} \right)\amalg (\eps_\text{br})
\end{equation}
	as an equality of sets.
	\item For each $(\G,\delta)$, we have that $\phi_{\G,\delta}(\partial \D(\G,\delta))$ is contained in a union of cells of smaller dimension.
\end{enumerate}
Then the proposition follows immediately \cite[Proposition A.2]{hatcher}.  Actually we will prove more in the process: we will analyze the contribution that each facet of a cell $\D(\G,\delta)$ makes to the boundary maps in the cellular chain complex. In order to talk about the chain maps, we shall assume that the vertices of each simplex $\Delta(\G,\delta)$ are equipped with a total ordering, and the faces of $\D(\G,\delta)$ are given the corresponding orientation.

Let us start by proving (1).  The point $\pt$ is mapped in $\dtn$ to the collapse of $\btno$.  Next, if $(G,l,m,w)$ is a tropical curve in $\mtno$ of theta type, then by Lemma~\ref{l:faces}\eqref{it:relint}, there is a unique decoration $\delta$ of $\G=(G,m,w)$ for which $l$ lies in the relative interior of $\D(\G,\delta)$.  
So equation~\eqref{eq:setequality} holds.  Now for any $(\G,\delta) \in \DTT_{2,n}$, consider the map
$$\D(\G,\delta)/\!\sim ~\lra \dtn$$
induced from identifying points in $\D(\G,\delta)$ in the same fiber of $\phi_{\G,\delta}$.  This map is an injection with compact domain and Hausdorff codomain (Observation~\ref{o:hausdorff}), so it is a homeomorphism onto its image.  
We already argued that $\sim$ is trivial on the interior of $\D(\G,\delta)$, so $\phi_{\G,\delta}$ restricts to a homeomorphism on the interior of $\D(\G,\delta)$.  This proves (1).

For part (2),  let $(\G=(G,m,w),\delta)\in \DTT_{2,n}$ and let $F$ be a facet of $\D(\G,\delta)$.  
By Lemma~\ref{l:faces} we have the following three cases.

{\bf Case 1:} $F$ is defined by $l(e) = 0$ for some $e\in E(G)$.  

First, if $\G/e \in \Tbr_{2,n}$, then in $\dtn$, we have that $F$ is collapsed to the 0-cell $\eps_{\text{br}}$.  Note that $F$ does not contribute to the boundary map, because the codimension of $\eps_{\text{br}}$ is too large.

So suppose that $\G/e \not \in \Tbr_{2,n}$.  Then $\G/e \in \TTh_{2,n}$ by Lemma~\ref{l:thetasleft}.  We have two possibilities.  
The first possibility is that $\G$ had no parallel edges, but $\G/e$ does, say edges $f_1, f_2\in E(G/e)$. 
See Figure~\ref{fig:createparallel}.

\begin{figure}
\begin{tikzpicture}[my_node/.style={fill, circle, inner sep=1.75pt}, scale=1.25]
\begin{scope}
\node[my_node] (A) at (-.5,0){};
\node[my_node] (B) at (.5,0){};
\draw[very thick] (A)--(B);
\draw[very thick] (A) to [out=-90,in=-90,looseness=1.25] (B);
\draw[very thick] (A) to [out=90,in=90,looseness=1.25] (B);
\draw[very thick] (-.35,.35) -- (-.45,.55);
\draw (-.55,.65) node {${\scriptstyle 1}$};
\draw[very thick] (.35,.35) -- (.45,.55);
\draw (.55,.65) node {${\scriptstyle 3}$};
\draw[very thick] (0,.43) -- (0,.65);
\draw (0,.75) node {${\scriptstyle 2}$};
\draw[very thick] (0,-.43) -- (0,-.65);
\draw (0,-.75) node {${\scriptstyle 4}$};
\draw (-.65,.2) node {${\scriptstyle e_0}$};
\draw (-.2,.53) node {${\scriptstyle e_1}$};
\draw (.2,.53) node {${\scriptstyle e_2}$};
\draw (.65,.2) node {${\scriptstyle e_3}$};
\draw (0,.1) node {${\scriptstyle e_4}$};
\draw (-.35,-.5) node {${\scriptstyle e_5}$};
\draw (.35,-.5) node {${\scriptstyle e_6}$};
\end{scope}
\begin{scope}[shift = {(3,0)}]
\node[my_node] (A) at (-.5,0){};
\node[my_node] (B) at (.5,0){};
\draw[very thick] (A)--(B);
\draw[very thick] (A) to [out=-90,in=-90,looseness=1.25] (B);
\draw[very thick] (A) to [out=90,in=90,looseness=1.25] (B);
\draw[very thick] (-.35,.35) -- (-.45,.55);
\draw (-.55,.65) node {${\scriptstyle 1}$};
\draw[very thick] (.35,.35) -- (.45,.55);
\draw (.55,.65) node {${\scriptstyle 3}$};
\draw[very thick] (0,.43) -- (0,.65);
\draw (0,.75) node {${\scriptstyle 2}$};
\draw[very thick] (0.5,0) -- (0.75,0);
\draw (.85,0) node {${\scriptstyle 4}$};
\end{scope}
\end{tikzpicture}
\caption{Two examples of cyclic theta types in $\Tsig_{2,4}$ for $\sig = (1~2~3~4)$, of form $\cyc{3}{1}$ and $\cycr{3}{0}$ respectively.  Our convention is that the edges are ordered in reading order, as shown on the left.  The type on the right is obtained by contracting edge $e_6$; note that this operation creates a new pair of parallel edges.}
\label{fig:createparallel}
\end{figure}
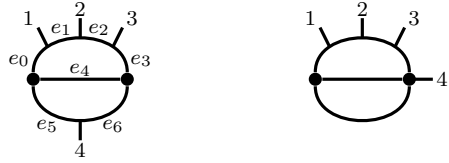

By abuse of notation, the edge in $G$ giving rise to $f_i$ will also be called $f_i$.  
Now $F$ obviously decomposes as the union of the two closed simplices $$F \cap \{l(f_1) \le l(f_2)\} \quad \text{ and }\quad F \cap \{l(f_1) \ge l(f_2)\}.$$  
In $\dtn$, each of these halves is identified with $\Delta(\G/e, f_1\le f_2)$, {\em but with opposite signs,} since the two identifications differ by a flip of parallel edges.  
Thus $F$ is sent to the image of $\Delta(\G/e, f_1\le f_2)$ in $\dtn$, which does indeed have codimension 1; however, its total contribution to the cellular boundary map is zero due to cancellation.

The other possibility is that contracting $e$ did not create a new parallel edges.  Then $F$ is identified with $\Delta(\G/e, \delta|_{\G/e})$ with a contribution of $\pm 1$ to the boundary map.  The sign simply depends on our choice of vertex-ordering for the two simplices.  

{\bf Case 2: $F$ is defined by $l(e) =l(e') = 0$}, for $e=_\delta e'$.
Then $F$ is sent to $\eps_{\text{br}}$ in $\dtn$, because contracting $e$ and $e'$ produces a vertex of positive weight.

{\bf Case 3: $F$ is defined by $l(e) = l(e')$}, for $e \le_\delta e'$.  Then $F$ is identified with $\D(\G, e=e')$ and the contribution to the boundary map is again $\pm 1$.

In each case, $F$ was identified with the closure of a cell of smaller dimension, so Proposition~\ref{p:cw}(2) is proved.

\end{proof}

\section{Homology of the cyclic theta complex}\label{s:cyclic}

Let $\G=(G,m,w)\in \TTh_{2,n}$.  We say $\G$ is a {\em cyclic theta} if there is a cycle in $\G$ passing through all the marked points.  If not, we say $\G$ is a {\em full theta}.  We let $\Tcyc_{2,n}$ and $\Tfull_{2,n}$ denote the set of all cyclic theta types and the set of all full theta types, respectively.
In this section, we compute the integral homology of the subcomplex of $\Delta_{2,n}$ supported on cyclic theta types.  In the next section, we will study the remaining full types, which will allow us to prove Theorems~\ref{t:1} and~\ref{t:3}.

\begin{defn}\label{d:cyclic}
  Let $C_{2,n}$ denote the subcomplex of $\dtn$
$$C_{2,n} = \zcell\cup\!\bigcup_{\G\in \Tcyc_{2,n}} \eps_{\G,\delta},$$
called the cyclic theta complex.
\end{defn}
\noindent This is a subcomplex of $\dtn$ because every contraction of a cyclic theta either is in the bridge locus or is again a cyclic theta type.
The next theorem computes the integral homology of $C_{2,n}$, and we will prove it in the rest of the section.  Our proof also explicitly constructs representatives for the nonzero homology classes in degree $n+1$.

\begin{thm}\label{t:cyclic}
For all $i\ge 0$, 
$$\Ht_i(C_{2,n},\ZZ)=\begin{cases}
({\ZZ}/{2\ZZ})^{\frac{(n-1)!}{2}} &\text{if } i = n+1, \\
0 &\text{else.}
\end{cases}$$
\end{thm}

\begin{proof}
Consider the $(n-1)!/2$ possible cyclic orderings of the set $\{1,\ldots,n\}$, i.e.~ways to place $n$ numbers around a circle. 
For example, when $n=4$ the three cyclic orderings are $(1234),~(1243),~(1324)$.
Suppose $\G=(G,m,w)\in \Tcyc_{2,n}$. Notice that $\G$ determines a unique cyclic ordering on $\{1,\ldots,n\}$.  Given a cyclic ordering $\sigma$, write $\Tsig_{2,n}$ for the types in $\Tcyc_{2,n}$ inducing $\sigma.$

Now suppose $e\in E(G)$.  As noted above, either $\G/e \in \Tbr_{2,n}$ or $\G/e$ is again a cyclic theta.  The crucial observation is that in the latter case, $\G$ and $\G/e$ determine the {\em same} cyclic ordering on $\{1,\ldots,n\}$.  Thus, for any cyclic ordering $\sigma$ of $\{1,\ldots,n\}$, we may define $\Csig_{2,n}$ to be the subcomplex on cells
$$\Csig_{2,n} = \zcell\cup\!\bigcup_{\G\in\Tsig_{2,n}} \eps_{\G,\delta}.$$
Then we have
$$C_{2,n} = \bigvee_{\sigma} \Csig_{2,n},$$
where the wedge is over all cyclic orderings of $\{1,\ldots,n\}$ and the identification is at the 0-cell $\zcell.$
Therefore 
$$\Ht_i(C_{2,n},\ZZ) \cong \bigoplus_\sigma \Ht_i (\Csig_{2,n}, \ZZ).$$

Therefore the following claim implies the theorem.

\begin{claim}\label{c:cyclic}
For any cyclic ordering $\sigma$ of $\{1,\ldots,n\}$, we have
$$\Ht_i(\Csig_{2,n},\ZZ)=\begin{cases}
{\ZZ}/{2\ZZ} &\text{if } i = n+1, \\
0 &\text{else.}
\end{cases}$$
\end{claim}

Now we prove the claim.  Clearly the choice of $\sigma$ does not matter; choose $\sig = (1\cdots n)$ once and for all.  Write $C^{\sigma,d}_{2,n}$ for the $d$-skeleton of $\Csig_{2,n}$.  For simplicity write 
\begin{equation}\label{eq:vi}
V_i = H_{i}(C^{\sigma,i}_{2,n},C^{\sigma,i-1}_{2,n};\ZZ).
\end{equation}

Consider the cells of $\Csig_{2,n}$.  A type $\G= (G,m)\in \Tsig_{2,n}$ is obtained from an unmarked graph of type I in Figure~\ref{f:genus2}, which we will once again call $\Theta$, by marking $n$ distinct points on it. (If the interior of an edge is marked then we consider it to be subdivided there.)  Thus, $G$ has $n+1,n+2,$ or $n+3$ edges, depending on whether 0, 1, or 2 of the vertices of $\Theta$ are marked.

Then by Lemma~\ref{l:faces}\eqref{it:dimdec}, the cell $\eps_{\G,\delta}$ has dimension $n-1,n,n+1,$ or $n+2$.  This already shows that 
$\Ht_i(\Csig_{2,n},\ZZ)=0$ for $i < n-1$.  
We are left to compute the homology of the complex of $\ZZ$-modules
\begin{equation}
0 \lra V_{n+2} \stackrel{\del_{n\!+\!2}}{\lra} V_{n+1} \stackrel{\del_{n\!+\!1}}{\lra} V_n \stackrel{\del_n}{\lra} V_{n-1} \lra 0
\label{eq:bdycx}
\end{equation}
with the modules $V_i$ as in~\eqref{eq:vi}.

We need to set some notation for the types $\G$ involved, and set a convention for ordering the vertices of the relevant simplices $\Delta(\G)$ and $\Delta(\G,\delta)$.  
Choose, once and for all, one of the two possible orientations of $\sigma$, say the oriented cyclic order $1,\ldots,n$.   We will say that $\G$ is of the form $(\eps_1, k_1, \eps_2, k_2)$, with $\eps_i \in \{0,1\}$ and $\eps_1+k_1+\eps_2+k_2=n$, if $\G$ is obtained by adding markings to a $\Theta$ graph such that, in the (oriented) cyclic order, there are $\eps_1$ marked points on a vertex of $\Theta$, then $k_1$ marked points on an edge, then $\eps_2$ marked points on the other vertex of $\Theta$, then $k_2$ marked points on another edge.  Obviously we have an equivalence of such forms via cyclic reordering: $(\eps_1,k_1,\eps_2,k_2)\sim(\eps_2,k_2,\eps_1,k_1).$  
For example, the types in Figure~\ref{fig:createparallel} have forms $(0,3,0,1)$ and $(0,3,1,0)$ respectively.

We adopt the convention that the edges of a type $\G$ of form $(\eps_1,k_1,\eps_2,k_2)$ are ordered as follows: we place the $k_1$ marked points in order, left to right, on the top edge of $\Theta$, and the $k_2$ marked points right to left on the bottom edge of $\Theta$; then we put the edges in reading order.   See Figure~\ref{fig:createparallel} for an example.
Note that our choice of ordering of the edges of $\G$ induces an order on the vertices of $\D(\G)$, when $\G$ is an automorphism-free type.  We will take this ordering throughout.

In keeping with the above convention, we will sometimes notate the forms 
$$(0,k_1,0,k_2)\qquad(1,k_1,0,k_2)\qquad(0,k_1,1,k_2)\qquad(1,k_1,1,k_2)$$
as
$$\cyc{k_1}{k_2}\qquad\qquad\cycl{k_1}{k_2}\qquad\qquad\cycr{k_1}{k_2}\qquad\qquad\cyclr{k_1}{k_2}$$
respectively, simply as a visual aid.  Furthermore we notate a specific type in the same way.  For example the type on the left in Figure~\ref{fig:createparallel} will be written
$\tbinom{123}{4}$
and it has form $\cyc{3}{1}.$

Furthermore, suppose $\G$ has a nontrivial automorphism.  That is, $\G$ has form $(\eps_1,k_1,\eps_2,k_2)$ and $k_1 = 0$ or $k_2=0$. Without loss of generality, suppose $k_2=0$.  Then $\G$ has a pair of parallel edges $e$ and $e'$, and $\G$ thus admits two possible decorations $\delta$, namely $e=e'$ or $e\le e'$.  Then by slight abuse of notation we will write $(\G,\delta)$ as having form $(\eps_1,k_1,\eps_2,=)$ or $(\eps_1,k_1,\eps_2,\le)$.  We extend this notation to the shorthand above.  For example $(1,k_1,0,\le)$ will be written $\cycl{k_1}{\le}$. 

Furthermore, we adopt the following convention on the order of the vertices of $\Delta(\G,\delta)$ in this case.  Write $e_0,\ldots, e_{n'}$ for the edges $E(G)-\{e,e'\}$ taken as they appear in the directed path compatible with the chosen orientation of $\sigma$.  See Figure~\ref{fig:new}.  So $n' = n -\eps_1 - \eps_2$.  
If the decoration $\delta$ is $e=e'$ then the vertices of $\D(\G,\delta)$ are $$\{\frac{1}{2}(l_e + l_{e'}), l_{e_0}, \ldots, l_{e_{n'}}\}$$ by Lemma~\ref{l:faces}\eqref{it:vertdec}, and we take them in that order.  If the decoration is $e \ge e'$ then the vertices of $\D(\G,\delta)$ are $$\{l_e, \frac{1}{2}(l_e + l_{e'}), l_{e_0},\ldots, l_{e_{n'}}\}$$ and we take them in that order.  
 We use analogous conventions if $\G$ has the form $(\eps_1, 0, \eps_2, k_2)$.  

In this way, we have chosen orientations for each of the cells of $\Csig_{2,n}$.  We will write down the boundary maps $\del_i$ in~\eqref{eq:bdycx} explicitly in terms of them.  Of course, they are just conventions chosen for the sake of explicit computation; they do not affect the Smith normal forms of $\del_i$.

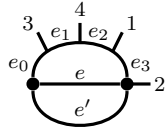
\begin{figure}
\begin{tikzpicture}[my_node/.style={fill, circle, inner sep=1.75pt}, scale=1.25]
\begin{scope}
\node[my_node] (A) at (-.5,0){};
\node[my_node] (B) at (.5,0){};
\draw[very thick] (A)--(B);
\draw[very thick] (A) to [out=-90,in=-90,looseness=1.25] (B);
\draw[very thick] (A) to [out=90,in=90,looseness=1.25] (B);
\draw[very thick] (-.35,.35) -- (-.45,.55);
\draw (-.55,.65) node {${\scriptstyle 3}$};
\draw[very thick] (.35,.35) -- (.45,.55);
\draw (.55,.65) node {${\scriptstyle 1}$};
\draw[very thick] (0,.43) -- (0,.65);
\draw (0,.8) node {${\scriptstyle 4}$};
\draw[very thick] (0.5,0) -- (0.75,0);
\draw (.85,0) node {${\scriptstyle 2}$};
\draw (-.65,.2) node {${\scriptstyle e_0}$};
\draw (-.2,.53) node {${\scriptstyle e_1}$};
\draw (.2,.53) node {${\scriptstyle e_2}$};
\draw (.65,.2) node {${\scriptstyle e_3}$};
\draw (0,.1) node {${\scriptstyle e}$};
\draw (0.03,-.25) node {${\scriptstyle e'}$};
\end{scope}

\end{tikzpicture}
\caption{The edges $e_0,\ldots,e_{n'}$ in the ordering compatible with the oriented cyclic order $\sigma = (1234)$.  Here $n' = n - 1 = 3.$}
\label{fig:new}
\end{figure}

\begin{lemma}\label{l:forms}
Suppose $\G$ is a decorated type that is a cyclic theta, with cyclic order $\sigma.$  Then we have the following case analysis for the possible forms of $(\G,\delta)$:
\begin{enumerate}
	\item $(\G,\delta)$ has one of the forms 
	$$\cyc{n}{\le} \quad  \cyc{n-1}{1} ~\cdots ~\cyc{\lceil n/2\rceil }{\lfloor n/2\rfloor }$$
	and in these cases $\eps_{\G,\delta}$ has dimension $n+2$;
	\item $(\G,\delta)$ has one of the forms
	$$\cyc{n}{=} \quad \cycl{n-1}{\le} \quad \cycl{n-2}{1} ~\cdots ~ \cycl{1}{n-2} \quad \cycl{\le}{n-1},$$
	and in these cases $\eps_{\G,\delta}$ has dimension $n+1$;
	\item $(\G,\delta)$ has one of the forms
	$$\cycl{n-1}{=}\quad \cycl{=}{n-1}\quad \cyclr{n-2}{\le} \quad \cyclr{n-3}{1} ~\cdots ~ \cyclr{\lceil n/2 \rceil - 1}{\lfloor n/2 \rfloor -1},$$
	and in these cases $\eps_{\G,\delta}$ has dimension $n$;
\item $(\G,\delta)$ has the form
	$$\cyclr{n-2}{=}$$
	and in this case $\eps_{\G,\delta}$ has dimension $n-1$.
\end{enumerate}
\end{lemma}
\begin{proof}
This is a straightforward case analysis; the dimensions of the cells were computed in Lemma~\ref{l:faces}\eqref{it:dimdec}.
\end{proof}
By counting we deduce the following.
\begin{lemma}\label{l:dims}
The modules $V_i$ in~\eqref{eq:vi} have ranks:
\begin{enumerate}
 \itemsep1em
\item $\rank V_{n+2} = {n+1\choose2}$,
\item $\rank V_{n+1} = 2{n+1\choose 2}$,
\item $\rank V_{n} = {n+1\choose 2} + n$,
\item $\rank V_{n-1} = n.$
\end{enumerate}
\end{lemma}
\noindent In this lemma, ${n+1\choose2}$ has its usual meaning ``$n\!+\!1$ choose $2$,'' not to be confused with the previous notation.
\begin{proof}[Proof of Lemma~\ref{l:dims}]
This is a routine count in each case. For example, part (1) is true because there are ${n+1\choose2}=n(n+1)/2$ ways to cut the numbers $1,\ldots,n$, arranged around a circle, into two segments (with length zero segments permitted).  Alternatively, the counts can be derived directly from Lemma~\ref{l:forms}.
\end{proof}

We can now compute the homology of~\eqref{eq:bdycx}.  The proof is by direct computation.  We compute the Smith normal forms over $\ZZ$ of each boundary map $\del_{n+2},\del_{n+1},$ and $\del_n$.  The boundary maps themselves are computed using the case analysis in the proof of Proposition~\ref{p:cw}.  We compute them with respect to the choices of orientations established earlier in the section.  The homology computations then immediately follow from the Smith normal forms.  From here on, $\SNF$ denotes Smith normal form, and $\Id$ and ${\mathbf 0}$ denote identity and zero matrices respectively.

\begin{claim}\label{c:n+2}
The normal form of $\del_{n+2}$ is
\begin{equation}\label{eq:n+2}
\SNF(\del_{n+2}) \,=\,\, \begin{blockarray}{ccccc}
 \BAmulticolumn{4}{c}{\scr {n(n+1)/2}}&\\
 \begin{block}{(cccc)c}
 1 &&&&\\
 & \ddots &&& \scr \scr {n(n+1)/2} \\
&& 1 && \\
 &&& 2 & \\
 \cline{1-4}
 &&&&\\&&&&\\
\BAmulticolumn{4}{c}{\boldsymbol 0} &\scr \scr {n(n+1)/2} \\
 &&&&\\ 
 \end{block}
 \end{blockarray}
\end{equation}
\end{claim}

\begin{claim}\label{c:n+1}
The normal form of $\del_{n+1}$ is
\begin{equation}\label{eq:n+1}
\SNF(\del_{n+1}) \,=\,\, \begin{blockarray}{ccc}
 \scr {n(n+1)/2} & \scr {n(n+1)/2} &  \\
  \begin{block}{(c|c)c}
 && \\ 
 \Id &\mathbf{0} & \scr \scr {n(n+1)/2} \\
 && \\
 \cline{1-2}
 \mathbf{0} & \mathbf{0}& \scr n\quad \\
 \end{block}
  \end{blockarray}
\end{equation}
\end{claim}

\begin{claim}\label{c:n}
The normal form of $\del_{n}$ is
\begin{equation}\label{eq:n}
\SNF(\del_{n}) \,=\,\, \begin{blockarray}{ccc}
 \scr n & \scr {n(n+1)/2} &\\
  \begin{block}{(c|c)c}
  &&\\
 \quad\Id\,\,\,\,\,\, &\mathbf{0} & \scr n \\
 &&\\
 \end{block}
  \end{blockarray}
\end{equation}
\end{claim}

Now we prove Claims~\ref{c:n+2},~\ref{c:n+1},~\ref{c:n}.  The boundary maps look different according to the parity of $n$, so we split our analysis into cases.  
\medskip

\begin{proof}[\bf Proof of Claim~\ref{c:n+2} for $n$ odd]
Let $n=2k+1$.
First, consider the the $n\times n$ square submatrix of $\del_{n+2}$ whose columns are indexed by all types of the form $\cyc{k+1}{k}$ and whose rows are indexed by all types of the form $\cycl{k}{k}$. Call this submatrix $A$.
We now show that $\det(A)\ne 0$, indeed, that $A$  
is an $n\times n$ double diagonal matrix, with two entries 1 in each column, corresponding to two permutations in $S_n$ of the same sign.  For example, when $n=5$, we have

$$A~=~\bordermatrix{ & \tbinom{123}{54} & \tbinom{234}{15} & \tbinom{345}{21} & \tbinom{451}{32} & \tbinom{512}{43} \cr
{\scriptstyle 1}\tbinom{23}{54} & 1 &\null&\null& 1 &\null\cr
{\scriptstyle 2}\tbinom{34}{15} &\null& 1 &\null&\null& 1 \cr
{\scriptstyle 3}\tbinom{45}{21} & 1 &\null& 1 &\null&\null\cr
{\scriptstyle 4}\tbinom{51}{32} &\null& 1 &\null& 1 &\null\cr
{\scriptstyle 5}\tbinom{12}{43} &\null&\null& 1 &\null& 1 }.$$

\noindent More generally, for each $i\in \{1,\ldots,n\}$, write
$$\G_i = \tbinom{i~\cdots~k+i}{i\!-\!1~\cdots~k+i+1}, \qquad \mathbf{H}_i = {\scriptstyle i}\tbinom{i+1~\cdots~ k+i}{i-1 ~\cdots~ k+i+1}.$$
Then the column of $A$ indexed by $\G_i$ has an entry 1 in row $\mathbf{H}_i$ and in row $\mathbf{H}_{k+i}$ (with indices modulo $n$), and no other nonzero entries.  This is because $\mathbf{H}_i$ is obtained by contracting edge $e_0$ of $\G_i$, i.e.~$\D(\mathbf{H}_i)$ is identified with facet $0$ of $\D(\G_i)$.  Similarly, $\mathbf{H}_{k+i}$ is obtained by contracting edge $e_{k+1}$ of $\G_i$, which gives sign $(-1)^{k+1}$; and then reversing the order of the $n+2 = 2k+3$ edges.  This reversal is a product of $k+1$ transpositions.  Hence the total contribution is $(-1)^{2(k+1)} = 1$.  
Then it follows that 
\begin{equation}\label{eq:A}
\det(A) = 1 + \sgn(1~k\!+\!1~n~k~\cdots~ k\!+\!2)= 2 \ne 0,
\end{equation}
where the permutation above is written in cyclic notation and is a $n$-cycle.

Then $\del_{n+2}$ has the following form, for suitable orderings of rows and columns:

\begin{equation}\label{eq:n+2fullform}
\bordermatrix{ &\cyc{k+1}{k} & \cyc{k+2}{k-1} & \cdots & \cyc{2k}{1} & \cyc{2k+1}{\le} \cr
\cycl{k}{k} & A &  & &\cdots &  \mathbf{0}\cr
\cycl{k+1}{k-1} & * & \Id & &  &\vdots  \cr
\quad\vdots &  & * & \Id &  &  \cr
\cycl{2k-1}{1} & \vdots &  & * & \Id &  \cr
\cycl{2k}{\le} & \mathbf{0} & \cdots &  &  & \Id\cr
\cline{2-6}
\cycr{k+1}{k-1} & -\Id & * &&\cdots &\mathbf{0} \cr
\cycr{k+2}{k-2} && -\Id & * &&\vdots \cr
\quad\vdots &&& -\Id & * & \cr
\cycr{2k}{\le} &\vdots &&&& -\Id \cr
\cyc{2k+1}{=} &\mathbf{0}&\cdots &&& \Id\cr
}
.
\end{equation}

\medskip

\noindent Here, each block represents an $n\times n$ submatrix whose rows and columns are indexed by types of the indicated forms.  The blocks notated $*$ are in fact signed $n\times n$ identity matrices up to row/column permutation.  

As a brief example of how to compute the entries in $\del_{n+2}$, we explain why the $n\times n$ submatrix with columns of form $\cyc{2k}{1}$ and rows of form $\cycr{2k}{\le}$ is $\mathbf{0}$.  Indeed, it is possible to contract an edge of a type of form $\cyc{2k}{1}$ to obtain a type of form $\cycr{2k}{\le}$ with two parallel edges $e$ and $e'$.  An example is shown in Figure~\ref{fig:createparallel}.  (In the figure, $n$ is even, but the essential point is the same.)  However, there are two possible choices of decoration for the new type: $e \ge e'$ and $e' \ge e$. So the corresponding codimension 1 cell occurs once with each orientations, and the total contribution is zero.  This was already argued in the proof of Proposition~\ref{p:cw}, and all of the other claimed entries of $\del_{n+2}$ similarly follow from the case analysis in that proof.

Now we compute $\SNF(\del_{n+2})$.  Consider the upper $(n+1)n/2 \times (n+1)n/2$ square submatrix of $\del_{n+2}$.  It is the submatrix above the line in~\eqref{eq:n+2}.  It is invertible since $A$ is, and by~\eqref{eq:A} its Smith normal form is $\diag(1,\ldots,1,2)$.  Therefore either $\SNF(\del_{n+2})$ is as claimed in~\eqref{eq:n+2}, or it is instead a $(n+1)n/2$ square {\em identity} matrix, padded with zeroes.  We show that the latter is not possible by exhibiting a torsion element in $V_{n+1}/\im(\del_{n+2})$.  Let $v$ be the sum of the first $(n-1)n/2$ columns of $\del_{n+2}$ in~\eqref{eq:n+2fullform}; call it $v$.  That is, $v$ is sum the column vectors corresponding to all types except those of the form $\cyc{2k+1}{\le}$.  By inspecting~\eqref{eq:n+2fullform}, we see $v\equiv 0\! \mod 2$.  The columns of $\del_{n+2}$ are linearly independent over $\QQ$, from which it follows that $v/2 \in V_{n+1}$ is not in the image of $\del_{n+2}$ considered as a $\ZZ$-module map.  This shows Claim~\ref{c:n+2} for $n$ odd, and gives an explicit representative for the nonzero homology class of $\Csig_{2,n}$ in degree $n+1$.
\end{proof}

\begin{proof}[\bf Proof of Claim~\ref{c:n+2} for $n$ even]
Let $n=2k$.  First, some temporary notation: write $\cycl{k}{k-1}'$, respectively $\cycl{k}{k-1}''$, for the types of the form $\cycl{k}{k-1}$ in which the marking on the left vertex of $\Theta$ is in $\{1,\ldots,k\}$, respectively in $\{k\!+\!1,\ldots,n\}$.  
For example, when $n=4$ the types of form $\cycl{2}{1}'$ are ${\scr 1}{23\choose 4}$ and ${\scr 2}{34\choose 1}$, and the types of form $\cycl{2}{1}''$ are ${\scr 3}{41\choose2}$ and ${\scr 4}{12\choose3}$.

Now let $A$ be the $(3n/2) \times (3n/2)$ square submatrix of $\del_{n+2}$ with columns indexed by all types of the form $\cyc{k+1}{k-1}$ or $ \cyc{k}{k}$ and rows indexed by all types of the form $\cycr{k}{k-1}$ or $\cycl{k}{k-1}'$.  For example, when $n=4$, we have
\begin{equation}\label{eq:B}
A=\bordermatrix{%
& \cyc{123}{4} & \cyc{234}{1} & \cyc{341}{2} & \cyc{412}{3} & \cyc{12}{43} & \cyc{23}{14}\cr
\cyc{12}{4}{\scriptstyle 3} &-1 & & & &1 & \cr
\cyc{23}{1}{\scriptstyle 4} & &-1 & & & &1 \cr
\cyc{34}{2}{\scriptstyle 1} & & &-1 & &-1 & \cr
\cyc{41}{3}{\scriptstyle 2} & & & &-1 & &-1 \cr
{\scriptstyle 1}\cyc{23}{4} &1 & & & & &1 \cr
{\scriptstyle 2}\cyc{34}{1} & &1 & & &-1 & 
}.
\end{equation}
We claim that $\det A=\pm 2$.  First, consider the $k$ types
$$\cyc{k\!+\!1 ~\cdots~ 1}{k~\cdots~2}\quad \cdots \quad \cyc{n~1 ~\cdots~k}{n\!-\!1~\cdots~k\!+\!1}.$$
Each of the columns they index has a single nonzero entry which is $\pm 1$, i.e.~in the $k$ rows defined by
$$\cyc{k\!+\!1 ~\cdots~ n}{k~\cdots~2}{\scriptstyle 1} \quad \cdots \quad \cyc{n~1 ~\cdots~k\!-\!1}{n\!-\!1~\cdots~k\!+\!1}{\scriptstyle k}.$$  So we may delete these $k$ rows and $k$ columns to obtain a matrix $A'$ with $\det A' = \pm \det A.$

Now we may check that $A'$ has the form
$$\begin{blockarray}{cc|c}
&\cyc{1 ~\cdots~ k\!+\!1}{n~\cdots~k\!+\!2}\quad \cdots \quad \cyc{k ~\cdots~n}{k\!-\!1~\cdots~1}&%
\cyc{1 ~\cdots~ k}{n~\cdots~k\!+\!1}\quad \cdots \quad \cyc{k ~\cdots~n\!-\!1}{k\!-\!1~\cdots~n}\\
\begin{block}{c(c|c)}
\begin{array}{c}\cyc{1 ~\cdots~ k}{n~\cdots~k\!+\!2}{\scriptstyle k\!+\!1}\\ \vdots \\  \cyc{k ~\cdots~n\!-\!1}{k\!-\!1~\cdots~1}{\scriptstyle n}\\ \null
\end{array}& (-\Id)^{k+1}&\Id \\
\cline{2-3}
\begin{array}{c}{\scriptstyle 1}\cyc{2 ~\cdots~ k\!+\!1}{n~\cdots~k\!+\!2}\\ \vdots \\  {\scriptstyle k}\cyc{k \!+\!1 ~\cdots~n}{k\!-\!1~\cdots~1}
\end{array}& \Id&\begin{array}{cccc}  & (-1)^k & & \\ & &\ddots& \\ &&&(-1)^k\\ -1 &&& \end{array} \\
\end{block}
\end{blockarray}
$$
For example, when $n=4$, $A'$ is the $4\times 4$ matrix obtained from~\eqref{eq:B} by deleting the middle two rows and columns.  The computations of each entry follow directly from our choice of ordering of the edges for each type.  Then we have
\begin{align}\label{eq:detb'}
\begin{split}
\det A' &= (-1)^{k(k+1)+(k-1)k+1}\sgn (\tau_1) + \sgn(\tau_2) \\&= (-1)^k + (-1)^k  =\pm 2,
\end{split}\end{align}
where $\tau_1 = (n~n\!-\!1~\cdots~k\!+\!1)$ is a $k$-cycle, and $\tau_2 = (1~k\!+\!1)(2~k\!+\!2)\cdots(k~n)$ is a product of $k$ disjoint transpositions.

Next, we can see that $\del_{n+2}$ has a square submatrix of the following form, again for suitable orderings of rows and columns:

\begin{equation}\label{eq:n+2even}
\begin{blockarray}{cccccc}
 &\cyc{2k}{\le}& \cyc{2k-1}{1}& \cdots & \cyc{k+2}{k-2}& \cyc{k+1}{k-1} ~\cyc{k}{k} \\
\begin{block}{c(cccc|c)}
\cycl{2k-1}{\le} & \Id &  & & &\cdots \quad \mathbf{0}\\
\cycl{2k-2}{1} &  & \Id &* & &\qquad \vdots \\
\vdots &  & & \Id & * &\\
\cycl{k+1}{k-2} &  &  & &\Id &*\quad \quad\\
\cline{2-6}
\begin{array}{c}\\[-.2cm] \cycr{k}{k-1} \\[.2cm] \cycl{k}{k-1}' \\ \null \end{array} & \begin{array}{c}\vdots\\[.2cm] \mathbf{0} \end{array} & \begin{array}{c} \\ \ldots \end{array}  & & &A\\
\end{block}
\end{blockarray}
\end{equation}
Note that the submatrix~\eqref{eq:n+2even} uses all columns of $\del_{n+2}$, and moreover by 
\eqref{eq:detb'} its normal form is a diagonal matrix with entries $(1,\ldots,1,2)$.
It follows that $\SNF(\del_{n+2})$ is either of the form claimed in~\eqref{eq:n+2}, or it is an ${n+1\choose2}\times {n+1\choose2}$ {identity} matrix, padded with zeroes.  Again we argue that the latter is not possible by exhibiting a nontrivial torsion element of $V_{n+1}/\im(\del_{n+2})$.  Consider the vector $v$ that is the sum of all of the columns of $\del_{n+2}$ except those indexed by types of the form $\cyc{2k}{\le}$.  Note $v\equiv 0\mod 2$.  This is because $v$ does not involve any rows indexed by $\cycl{2k-1}{\le}, \cycr{2k-1}{\le},$ or $\cyc{2k}{=}$, and every row other than these $3n$ rows has exactly two $\pm 1$ nonzero entries.
So we have produced a nontrivial torsion element of $V_{n+1}/\im(\del_{n+2})$, namely $v/2$.  This constructs a representative for a nonzero homology class in degree $n+1$, and proves the claim.
\end{proof}

\begin{proof}[\bf Proof of Claim~\ref{c:n+1} for $n$ odd]
Let $n=2k+1$.
The matrix for $\del_{n+1}$ has the following form. Again, each symbol represents an $n\times n$ submatrix, with rows and columns indexed by the types of the indicated forms.  
\begin{equation*}
\begin{blockarray}{ccccccccccc}
& \cycl{k}{k} & \cycl{k+1}{k-1} & \cdots & \cycl{2k-1}{1} & \cycr{k+1}{k-1} &\cdots & \cycr{2k-1}{1} & \cycl{2k}{\le} & \cycr{2k}{\le} & \cyc{2k+1}{=}\\[.2cm]
\begin{block}{c(cccc|ccc|ccc)}
\cyclr{k}{k-1} & B & * && & \Id &&&&&\\
\vdots & & * & * && * & \Id &&&& \\[.2cm]
\cyclr{2k-2}{1} & & &* & * & & * &\Id &&&\\[.2cm]
\cline{2-11}
&&&&&&&&&&\\[-.25cm]
\cyclr{2k-1}{\le} &&&&&&&& \Id & \Id & \\[.2cm]
\cycl{2k}{=} &&&&&&&& \Id & & \!\!\!-\!\Id \\[.2cm]
\cycr{2k}{=} &&&&&&&&& \Id & \Id \\
\end{block}
\end{blockarray}
\end{equation*}
Here $B$ is a double diagonal matrix and the matrices $*$ are signed identity matrices up to row/column permutation.  After column operations preserving Smith normal form, we may assume instead that $B$ and all the matrices $*$ are zero.  Moreover the lower right $3n\times 3n$ submatrix has Smith normal form 
$$\left(\begin{array}{ccc}
\Id & & \\
& \Id & \\
& & \mathbf{0}
\end{array}\right).$$
This shows Claim~\ref{c:n+1} for $n$ odd.
\end{proof}

\begin{proof}[\bf Proof of Claim~\ref{c:n+1} for $n$ even]
Let $n=2k.$
The matrix for $\del_{n+1}$ has the following form:
\begin{equation*}
\begin{blockarray}{cccccccccc}
& \cycl{k}{k-1} & \cdots & \cycl{2k-2}{1} & \cycr{k}{k-1} &\cdots & \cycr{2k-2}{1} & \cycl{2k-1}{\le} & \cycr{2k-1}{\le} & \cyc{2k}{=}\\[.2cm]
\begin{block}{c(ccc|ccc|ccc)}
\cyclr{k-1}{k-1} &  B' &  &  &    B'' &&&&&\\[-.1cm]
\vdots & * & * &  & * & \Id &&&& \\[.2cm]
\cyclr{2k-3}{1} & &* &* & &* &\Id & & &\\[.2cm]
\cline{2-10}
&&&&&&&&&\\[-.25cm]
\cyclr{2k-2}{\le} &&&&&&& \!\!\!-\!\Id & \Id & \\[.2cm]
\cycl{2k-1}{=} &&&&&&& \Id & & -\Id \\[.2cm]
\cycr{2k-1}{=} &&&&&&&& \Id & -\Id\\
\end{block}
\end{blockarray}
\end{equation*}
where $B'$ and $B''$ are the $k\times n$ matrices 
$$\left(\begin{array}{c|c} &\\[-.1cm] \,\,\,(-\Id)^k\,\,\, & \,\,\,\Id\,\,\, \\[-.1cm] & \\\end{array}\right)\quad\text{and}\quad\left(\begin{array}{c|c} &\\[-.1cm] \,\,\,\Id\,\,\, & \,\,\,(-\Id)^k\,\,\, \\[-.1cm] & \\\end{array}\right) $$
and the matrices $*$ are signed $n\times n$ identity matrices up to row/column permutation.  After column operations preserving Smith normal form, we may assume instead that $B'$ and all the matrices $*$ are zero, and that 
$$B'' = \left(\begin{array}{c|c} &\\[-.1cm] \,\,\,\mathbf{0}\,\,\, & \,\,\,(-\Id)^k\,\,\, \\[-.1cm] & \\\end{array}\right). $$
Moreover the lower right $3n\times 3n$ submatrix has Smith normal form 
$$\left(\begin{array}{ccc}
\Id & & \\
& \Id & \\
& & \mathbf{0}
\end{array}\right).$$
Putting these statements together shows Claim~\ref{c:n+1} for $n$ even.
\end{proof}

\begin{proof}[\bf Proof of Claim~\ref{c:n} for all $n$]
This claim is clear from the fact that $\del_n$ restricted to types of the form $\cyclr{n-2}{\le}$ is $\mathrm{Id}_{n\times n}$.  
\end{proof}

Claims~\ref{c:n+2},~\ref{c:n+1}, and~\ref{c:n}, which we just proved, imply Claim~\ref{c:cyclic} immediately, and Claim~\ref{c:cyclic} proves Theorem~\ref{t:cyclic}.

\end{proof}

\section{The $\ZZ$-homology of $\mtno$ vanishes in codimension $>1$}\label{s:mainproof}

In this section, we will study the final piece of the complex $\dtn$, the locus of {\em full theta types}.  Recall that a type $\G\in \TTh_{2,n}$ is {\em full} if its marked vertices do not lie on a single cycle.  Equivalently, $\G$ is obtained from the unmarked graph $\Theta$ by adding markings such that at least one marked point lands on the interior of each of the three edges of $\Theta.$  

Suppose $\G$ is a theta type.  Let us say it has form $(\eps_1,\eps_2,k_1,k_2,k_3)$, with $1 \ge \eps_1 \ge \eps_2 \ge 0$ and $k_1 \ge k_2 \ge k_3 \ge 0$, if it is obtained from the unmarked graph $\Theta$ by placing $\eps_1$ and $\eps_2$ markings on each of the two vertices and $k_1,k_2,$ and $k_3$ markings on the interiors of the three edges.
Thus $\eps_1+\eps_2+k_1+k_2+k_3 = n$.

Let $F_{2,n}$ be the closure in $\dtn$ of the locus of full thetas, i.e.
$$F_{2,n} = \overline{\bigcup_{\G\in\Tfull_{2,n} } \eps_\G}.$$
Thus $F_{2,n}$ is the subcomplex of $\dtn$ whose cells are all possible contractions of full theta types, and it inherits a CW structure from $\dtn$.  Then $C_{2,n} \cup F_{2,n} = \dtn$.  Now
Theorem~\ref{t:cyclic} implies that $\Ht_i(C_{2,n};\QQ) = 0$ for all $i$ and that $\Ht_i(C_{2,n};\ZZ) = 0$ for all $i\le n$.  Then we have 
\begin{equation}\label{eq:overQ}
\tilde{H}_*(\dtn;\QQ) \cong \tilde{H}_*(\dtn/C_{2,n};\QQ) \cong \tilde{H}_*(F_{2,n} / (F_{2,n}\cap C_{2,n}); \QQ).
\end{equation}
and, for $0 \le i \le n$,
\begin{equation}\label{eq:overZ}
\tilde{H}_i(\dtn;\ZZ) \cong \tilde{H}_i(\dtn/C_{2,n};\ZZ) \cong \tilde{H}_i(F_{2,n} / (F_{2,n}\cap C_{2,n}); \ZZ). 
\end{equation}

So we study $F_{2,n}$ and $F_{2,n}/(F_{2,n}\cap C_{2,n})$ now. 
The next lemma describes the cells in $F_{2,n}$.
\begin{lemma}\label{l:fforms}
Suppose $\eps_{\G,\delta}$ is a cell in $F_{2,n}$ other than $\eps_{\mathrm{br}}$.  Then we have the following case analysis for the possible forms of $(\G,\delta)$:

\begin{enumerate}
	\item \label{it:fforms1} $(\G,\delta)$ is a full theta of the form $(0, 0, k_1, k_2, k_3)$, and in this case $\eps_\G$ has dimension $n+2$;
	\item \label{it:fforms2} $(\G,\delta)$ is a full theta of the form $(1, 0, k_1, k_2, k_3)$ or a cyclic theta of the form $\cycl{k_1}{k_2}$ for $k_1,k_2>0$, and in these cases $\eps_\G$ has dimension $n+1$;
	\item \label{it:fforms3} $(\G,\delta)$ is a full theta of the form $(1, 1, k_1, k_2, k_3)$ or is a cyclic theta of the form $\cyclr{k_1}{k_2}$ with $k_1,k_2>0$, or is a cyclic theta of the form $\cyclr{n-2}{\le}$, and in these cases $\eps_\G$ has dimension $n$;
	\item \label{it:fforms4} $(\G,\delta)$ is a cyclic theta of the form $\cyclr{n-2}{=}$, and in this case $\eps_\G$ has dimension $n-1$.
	\end{enumerate}
\end{lemma}

\begin{proof}
The cells in $F_{2,n}$, apart from the unique $0$-cell $\zcell$, are indexed by decorated types $\eps_{\G,\delta}$ where $\G\in \TTh_{2,n}$ is obtained from a full theta by contraction (see Lemma~\ref{l:thetasleft}).  Once again, regard $\G$ as obtained from the unmarked graph $\Theta$ by adding markings.  Then $\G\in \TTh_{2,n}$ is a contraction of a full theta type if and only if the number of edges of $\Theta$ whose interiors remain unmarked does not exceed the number of vertices of $\Theta$ that are marked.  For example, a cyclic type of the form $\cycl{n-1}{\le}$ fails the criterion above; the point is that even if the single marking on the 3-valent vertex is moved onto an edge, the resulting type is still not full.

The lemma follows from the observation above by a straightforward case analysis.  The dimensions were computed in Lemma~\ref{l:faces}.

\end{proof}

\begin{prop}\label{p:f}
For all $i\le n$, we have $$\Ht_i(F_{2,n}/(F_{2,n}\cap C_{2,n});\ZZ) = 0.$$
\end{prop}

\begin{proof}
Write $\del_i$ for the boundary maps of the cellular chain complex for $F_{2,n}/(F_{2,n}\cap C_{2,n})$.  From Lemma~\ref{l:fforms}\eqref{it:fforms4}, we see that $F_{2,n}/(F_{2,n}\cap C_{2,n})$ has no cells of dimension less than $n$.  So we are reduced to proving the statement for $i=n$.  Furthermore, by Lemma~\ref{l:fforms}\eqref{it:fforms3}, the $n$-cells of $F_{2,n}/(F_{2,n}\cap C_{2,n})$ correspond to all types obtained from $\Theta$ by marking both vertices once and marking the three edges with $k_1\ge k_2\ge k_3 \ge 1$ markings.  Then  $k_1+k_2+k_3 = n-2$ and $k_1 \le n-4$.  Now suppose $\G$ is of the form we just described.  
To show the Proposition for $i=n$, 
we want to show that $\eps_\G \in \im \del_{n+1}$.  
We induct downward on $k_1$.  Suppose $k_1 = n-4$. Then $k_2=k_3=1$.  Consider the type $\G'$ obtained from $\G$ by moving either of the two marked points on the vertices of $\Theta$ to the interior of the $k_1$-marked edge of $\Theta$.  We claim that $\del_{n+1}(\eps_{\G'}) = \pm \eps_{\G}$. This is because starting with $\G'$ and moving the single marking on the interior of either once-marked edge of $\Theta$ to the unmarked  vertex of $\Theta$ produces a cyclic type and hence no contribution to the boundary map.  Of course, any other edge contraction produces a repeating type and again no contribution to the boundary map.

Next, suppose inductively we have already shown that $\eps_{\mathbf H} \in \im \del_{n+1}$ for any type ${\mathbf H}$ of the form $(1,1,m_1,m_2,m_3)$ such that $m_1\ge m_2 \ge m_3 \ge 1$ and $m_1 > k_1$.  Again, we want to show $\eps_\G \in \im \del_{n+1}$.  Consider the type $\G'$ gotten by moving either one of the marked points on the 3-valent vertices of $\G$ to the interior of the first edge of $\Theta$, i.e.~the one supporting $k_1$ markings.  Then $\G'$ has the form $(1,0,k_1+1, k_2, k_3)$.  Then 
$$\del_{n+1} (\eps_{\G'}) = \pm \eps_\G \pm \eps_{\G_2} \pm \eps_{\G_3}$$
for types $\G_2$ and $\G_3$ of the forms $(1,1,k_1+1, k_2-1, k_3)$ and $(1,1,k_1+1, k_2, k_3-1)$.  But by the inductive hypothesis, both $\eps_{\G_2}, \eps_{\G_3} \in \im \del_{n+1}$, so $\eps_{\G} \in \im \del_{n+1}$ as well.
\end{proof}

Now we prove Theorems~\ref{t:1} and~\ref{t:3}.  Recall that a path-connected space is called $k$-connected if its homotopy groups $\pi_i$ vanish for all $1\le i \le k$.
 
\begin{proof}[Proof of Theorem~\ref{t:1}]  From~\eqref{eq:overZ}, Proposition~\ref{p:f}, and the fact that $\mtno$ and $\dtn$ have the same homotopy type (Observation~\ref{o:hausdorff}), it follows that the reduced $\ZZ$-homology of $\mtno$ also vanishes in degrees up to $n$.  Furthermore $\mtno$ is clearly 1-connected, since it is homotopy equivalent to a CW complex $\dtn$ with one 0-cell and no 1-cells.  Hence the vanishing of $\ZZ$-homology in degree up to $n$ implies that $\mtno$ is $n$-connected \cite[VII.10.9]{bredon}.
\end{proof}

\begin{proof}[Proof of Theorem~\ref{t:3}]
This is a computational result.  We carried out the computations in {\texttt sage} \cite{sage} and the code will be made available.  
If $n\le 3$, the theorem follows from a direct computation in sage.  The spaces $\mtno$ for $n\le 3$ are small enough that no reductions are necessary.  
For $n\ge 4$, the computations rely on the identification of $\QQ$-homology~\eqref{eq:overQ} which is computationally very significant.  We start by building the top boundary matrix $\del_{n+2}$ for the cellular homology of $F/(F\cap C)$.  The rows and columns of $\del_{n+2}$ are indexed by full theta types of dimensions $n+1$ and $n+2$, respectively. Then we use sage to compute the rank of $\del_{n+2}$.  The computation is carried out over $\QQ$ as opposed to $\ZZ$ because it is much faster.  The data of the rank of $\del_{n+2}$, along with the counts of the full theta types of codimension 0, 1, and 2, give the results in Table~\ref{table:upto8}.  The counts of the full theta types are explained in the next section, where they are used to prove a formula for the Euler characteristic of $\mtno$.
\end{proof}

\section{The Euler characteristic of $\mtno$}\label{s:euler}

In the final section of this paper we prove a simple formula for the Euler characteristic of $\mtno$. This calculation is independent of Section~\ref{s:mainproof}.

\begin{prop}
\
\begin{enumerate}
	\item If $n=0, n=1,$ or $n=3$, we have $\chi(\mtno) = 1.$
	\item If $n=2$ then we have $\chi(\mtno) = 2.$
\end{enumerate}
\end{prop}
\begin{proof} This follows immediately from Theorem~\ref{t:3}.
\end{proof}

\begin{thm}\label{t:euler}
 For all $n\ge 4$, we have
	$$\chi(\mtno) = 1 + (-1)^n\cdot n!/12.$$
\end{thm}

\begin{proof}
From Observation~\ref{o:hausdorff}, we may compute $\chi(\dtn)$ instead.  This complex has a zero cell $\zcell$, cells corresponding to full thetas types, and cells corresponding to cyclic theta types. From Lemma~\ref{l:dims}, we see that the total signed contribution of the cells corresponding to cyclic theta types is zero.  The cell $\zcell$ contributes 1 to $\chi(\dtn)$, so we are left to count, with sign, the number of full theta types with $n$ marked points.

First, given a partition $p\vdash m$, let $\alpha(p)$ denote the largest number of parts in $p$ of the same size.  For example, if $p=(4,3,3)\vdash 10$, then $\alpha(p) =2$.  We also fix the following notation: we write $\Theta(k_1,k_2,k_3)$ for the graph obtained from $\Theta$ by adding $k_i$ internal vertices to each edge $e_i$ of $\Theta$ for $i=1,2,3$.
In the notation of Section~\ref{s:mainproof}, it is of type $(0,0,k_1,k_2,k_3)$.  

Let us start by counting the top-dimensional full theta types.  By Lemma~\ref{l:fforms}, these are all full theta types of the form $(0,0,k_1,k_2,k_3)$  with $k_1\ge k_2\ge k_3 > 0$.  Let us fix one such choice of $k_1,k_2,$ and $k_3$ and count the  full theta types $\G$ of that form.  There are $$\frac{n!}{2\cdot \alpha(k_1,k_2,k_3)!}$$ such types, because there are then $n!$ ways to assign the markings, at which point we have overcounted by a factor of $|\Aut(\Theta(k_1,k_2,k_3))| = 2\alpha(k_1,k_2,k_3)!$. 
Thus, letting
\begin{equation}\label{eq:fdef}
f(n) = \sum_{(k_1,k_2,k_3)\vdash n} \frac{1}{\alpha(k_1,k_2,k_3)!},
\end{equation}
we have just shown $n!\cdot f(n)/2$ counts the number of full theta types of top dimension, i.e.~dimension $n+2$.

Next, the full theta types in codimension 1 are of the form $(1,0,k_1,k_2,k_3)$ for $k_1+k_2+k_3 = n-1$ (Lemma~\ref{l:fforms}).  By an analogous argument to the one above, there are $n!/\alpha(k_1,k_2,k_3)!$ full theta types of the form $(1,0,k_1,k_2,k_3)$.  (The factor of 2 disappears because exactly one of the two vertices of $\Theta$ is now marked, so they can no longer be exchanged.)  Thus the number of full theta types of codimension 1 is $n! \cdot f(n-1).$

Finally, an analogous argument shows that the number of full theta types of codimension 1 is $n!\cdot f(n-2)/2.$  Adding, we have
$$\chi(\dtn) = 1+ (-1)^n (n!) \left(\frac{f(n)}{2} - f(n\!-\!1) + \frac{f(n-2)}{2}\right).$$
Then the theorem follows from the following claim:
\begin{claim}\label{c:1/12}
Consider the function $f\colon \ZZ\ra \QQ$ defined in \eqref{eq:fdef}. For each $n \ge 4$, we have
$$\frac{f(n)}{2} - f(n\!-\!1) + \frac{f(n-2)}{2} = \frac{1}{12}.$$
\end{claim}
\begin{proof}[Proof of Claim~\ref{c:1/12}]
We  calculate $f$ as follows.  Consider the largest part $k_1$ of a partition $(k_1,k_2,k_3)\vdash n$.  The possible values of $k_1$ are $$\lceil \tfrac{n}{3}\rceil,\ldots,n\!-\!2.$$
First, if $k_1\ge \lceil \tfrac{n}{2} \rceil$, then the possible partitions with $k_1$ as largest part are
$$(k_1, n\!-\!k_1\!-\!1, 1),\quad (k_1, n\!-\!k_1\!-\!2,2),\quad \cdots \quad(k_1, \lceil \tfrac{n-k_1}{2} \rceil, \lfloor \tfrac{n-k_1}{2} \rfloor).$$
The assumption $k\ge \lceil \tfrac{n}{2} \rceil$ implies that $k_1 > k_2$ in each case.  Furthermore, the last partition has $k_2 = k_3$ if and only if $n-k_1$ is even.  
It follows that for each $k_1 = \lceil \tfrac{n}{2} \rceil, \ldots, n-2$, 
the partitions with largest part $k_1$ contribute $\tfrac{n-k_1-1}{2}$ to $f(n)$.

Next, if $\lceil \tfrac{n}{3} \rceil \le k_1 \le \lceil \tfrac{n}{2}\rceil \!-\! 1$, then the partitions $(k_1,k_2,k_3)\vdash n$ with $k_1$ as largest part are
\begin{equation}\label{eq:k1_2}
(k_1, k_1,n\!-\!2k_1), \quad \ldots,\quad (k_1, \lceil \tfrac{n\!-\!k_1}{2}\rceil, \lfloor \tfrac{n\!-\!k_1}{2}\rfloor).
\end{equation}
Of these, the first one has $k_1 = k_2$.  The last one has $k_2=k_3$ if and only if $n\!-\!k_1$ is even.  If these two partitions are different, then the total contribution to $f(n)$ of the partitions in~\eqref{eq:k1_2} is $\tfrac{3k_1-n}{2}$.   If they are the same, we must have $k_1 = n/3$ in which case the list in~\eqref{eq:k1_2} consists solely of the partition $(k_1,k_1,k_1)$ and it is counted with weight $1/6$.  

Putting together the two cases we just analyzed, we have
\begin{equation}\label{eq:f1}
f(n)=\delta \,\,+\,\,  \sum_{k_1 = \lceil n/3 \rceil}^{\lceil n/2\rceil \!-\! 1}\!\!\frac{3k_1-n}{2} \,\,+\,\, \sum_{k_1 = \lceil n/2\rceil}^{n\!-\!2} \!\!\frac{n-k_1 -1}{2}
\end{equation}
where $$\delta = \begin{cases} 1/6 &\text{ if } 3\,|\,n,\\0&\text{ otherwise.}\end{cases}$$

From this, we can verify directly that
\begin{equation}\label{eq:f2}
f(n) = \frac{1}{12}n^2 -\frac{1}{4}n + \frac{1}{6}
\end{equation}
e.g.~by checking, for each residue class modulo 6, that expressions~\eqref{eq:f1} agrees with~\eqref{eq:f2}.  Then Claim~\ref{c:1/12} follows immediately from~\eqref{eq:f2}.  This finishes the proof of Theorem~\ref{t:euler}.
\end{proof}

\end{proof}

\end{document}